\documentclass[10pt]{amsart}

\usepackage{amssymb}\usepackage{amsmath} 
\usepackage[T1]{fontenc}
 \usepackage{textcomp}
\usepackage[bookmarks=false,pdfstartview={FitBH},colorlinks={true},
pdfview={FitBH}]{hyperref}
\usepackage{orcidlink}
\usepackage[numbers]{natbib}
\usepackage{newpxtext,newpxmath}
\usepackage{enumitem}

\newtheorem{theorem}{Theorem}
\newtheorem{lemma}[theorem]{Lemma}

\newtheorem{question}[theorem]{Question}

\newtheorem{corollary}[theorem]{Corollary} 
\theoremstyle{definition}
\newtheorem{definition}[theorem]{Definition}

\theoremstyle{remark}

\newcommand{\converge}{\!\downarrow}

\newcommand{\imply}{\Rightarrow}

\newcommand{\Fraisse}{Fra{\"{i}}ss{\'{e}}}

\newcommand{\om}{\omega}

\newcommand{\sse}{\subseteq}
\newcommand{\contains}{\supseteq}

\newcommand{\mK}{\mathcal{K}}

\DeclareMathOperator{\Lev}{Lev}
\DeclareMathOperator{\AC}{AC}
\DeclareMathOperator{\Con}{Con}

\begin{document}

\author{Peter Cholak}
\address{University of Notre Dame}  
\email{cholak@nd.edu}
\author{Natasha Dobrinen}
\email{ndobrine@nd.edu}
\author{Charles McCoy}
\email{cmccoy@holycrossusa.org}



\title{The Henson graphs: colorings and codings}

\thanks{Cholak was partially supported by NSF-DMS-2502292 and Erwin Schr{\"o}dinger Institute. ORCID (cholak): \href{https://orcid.org/0000-0002-6547-5408}{\orcidlink{0000-0002-6547-5408}{0000-0002-6547-5408}}. Dobrinen was partially supported by NSF-DMS-2300896 and Erwin Schr{\"o}dinger Institute.
ORCID (Dobrinen): \href{https://orcid.org/0000-0001-6360-0001}{\orcidlink{0000-0001-6360-0001}{0000-0001-6360-0001}}.}

\maketitle

\begin{abstract}
   By recent work of \citet{Dobrinen_ICM} and \citet{Balko7} we know that every finite $G$ in the Henson graph $\mathbb{H}_{n+1}$ (the universal ultrahomogeneous $(n+1)$-clique free graph) has  exact finite big Ramsey degree $k({G,n})$. That is, there is a positive integer $k({G,n})$ such that for each finite coloring $C$ of the copies of $G$ in $\mathbb{H}_{n+1}$, there is $\tilde{\mathbb{H}}$, a substructure of $\mathbb{H}_{n+1}$ and isomorphic to $\mathbb{H}_{n+1}$, such that in $\tilde{\mathbb{H}}$ at most $k({G,n})$ colors are used on the copies of $G$ in $\tilde{\mathbb{H}}$. Moreover, for exactness,  for some coloring and all corresponding  $\tilde{\mathbb{H}}$, all $k({G,n})$ colors are needed. The ultimate result here is that if $|G|\geq 2$, then there is a finite computable coloring $C$ such that,  for  all such  $\tilde{\mathbb{H}}$, we have that $\tilde{\mathbb{H}}$ computes  $\emptyset^{(|G|-1)}$ (and hence the halting set).
\end{abstract}

\section{Introduction}

\begin{definition}
     Given a countable infinite structure $\mathbb{B}$ and a finite substructure $\mathbb{A}$ of $\mathbb{B}$. 
     Let $\mathcal{R}_{\mathbb{A},\mathbb{B}}(C,\tilde{\mathbb{H}})$ be the statement:  ``$C$ is a finite coloring of all copies of $\mathbb{A}$ in $\mathbb{B}$,  $\tilde{\mathbb{H}}$ is a substructure of $\mathbb{B}$, isomorphic to
     $\mathbb{B}$,
     %$\tilde{\mathbb{H}}$, 
     and all copies of $\mathbb{B}$ within $\tilde{\mathbb{H}}$ have the same fixed color." We say $\mathbb{A}$ has the \emph{Ramsey Property} in $\mathbb{B}$ iff, for all $C$, there is an $\tilde{\mathbb{H}}$ such that $\mathcal{R}_{\mathbb{A},\mathbb{B}}(C,\tilde{\mathbb{H}})$. 
\end{definition}

Here we will use various versions of  Henson graphs for $\mathbb{B}$.  Recall, for each $n\ge 2$, the $(n+1)$-st Henson graph, $\mathbb{H}_{n+1}$, is the universal, ultrahomogeneous, $(n+1)$-clique free, graph (for details, see Section~\ref{H3}). It is also the \Fraisse\ limit of  the class of all finite $(n+1)$-clique free graphs. It is well-known that nodes 
%and (some) anticliques 
have the Ramsey Property in any Henson graph; this is due to work of Komjath and R\"{o}dl for $n=2$ \cite{Komjath/Rodl86}, and El-Zahar and Sauer for all $n\ge 3$ \cite{El-Zahar/Sauer89}
(see Theorem~\ref{Indivisibility}).
% {\color{red}  
% and \ref{frd1}, respectively,
% and pairs in  $\mathbb{H}_3$, see Theorem~\ref{frd3}.   (ND:  This is not stated right, because the far anticliques etc have one color in the expanded Henson graph,  that is, in the diary, but definitely not in the Henson graph.  Do you want to reword this or just take it out?)}
Throughout the paper $\tilde{\mathbb{H}}$ will always be a copy of some $\mathbb{H}_{n+1}$ and all colorings considered are finite colorings.

The theorem that nodes have the Ramsey Property in $\mathbb{H}_{n+1}$ is often stated as that $\mathbb{H}_{n+1}$ is {\em indivisible}. If we divide the nodes of $\mathbb{H}_{n+1}$ into finitely many  pieces, then one of the pieces must contain a copy of $\mathbb{H}_{n+1}$.  Indivisibility is very important to our work; it provides the needed local structure of $\mathbb{H}_{n+1}$ and is discussed in Section~\ref{indi}. A proof of indivisibility for $\mathbb{H}_{n+1}$ appears in Section~\ref{proof:Indivisibility}.  This proof is done in the style of a forcing construction and hence yields computability theoretic corollaries and in addition might provide more clarity than other published proofs. 
% {\color{red}  (ND:  Do you want to say something about this being a simpler new proof with a stronger computability-theoretic conclusion?)}
% Other than this we will not discuss the proofs of these theorems mentioned above here. 
% {\color{red}  (ND:  I don't understand this sentence, because we are going to prove Theorems 5.1, 7.3, and 8.3)}

When we say these theorems are true, we mean that they hold in the two sorted structure $(\mathbb{N},2^{\mathbb{N}},+,\times,0,1<)$. In this realm, $C$ and $\tilde{\mathbb{H}}$ are sets, i.e., members of $2^{\mathbb{N}}$, and $\mathcal{R}_{\mathbb{A},\mathbb{B}}(C,\tilde{\mathbb{H}})$ is an arithmetic formula (no quantification over sets), given the set $\mathbb{B}$, and the $\Pi^1_2$ sentence $\forall C \exists \tilde{\mathbb{H}}\mathcal{R}_{\mathbb{A},\mathbb{B}}(C,\tilde{\mathbb{H}})$ is true, i.e., $$(\mathbb{N},2^{\mathbb{N}},+,\times,0,1<) \models \forall C \exists \tilde{\mathbb{H}}\mathcal{R}_{\mathbb{A},\mathbb{B}}(C,\tilde{\mathbb{H}}).$$
Here we want to consider $\mathcal{P} \subseteq 2^{\mathbb{N}}$ where
$$(\mathbb{N},\mathcal{P},+,\times,0,1<) \models \forall C \exists \tilde{\mathbb{H}}\mathcal{R}_{\mathbb{A},\mathbb{B}}(C,\tilde{\mathbb{H}}).$$  Here will restrict ourselves to only those nonempty $\mathcal{P}$ which are closed under Turing reducibility (this will be important later when we use Definition~\ref{equiv}). For shorthand, we will just say: ``$\mathcal{P}$ realizes $\forall C \exists \tilde{\mathbb{H}}\mathcal{R}_{\mathbb{A},\mathbb{B}}(C,\tilde{\mathbb{H}})$." 
% Also we will always consider $C$ as a coloring. Since being a coloring is an arithmetic property we assume if $C \in \mathcal{P}$ is not a coloring we will assume that $\mathcal{P}$ realizes $\forall \tilde{\mathbb{H}}\mathcal{R}_{\mathbb{A},\mathbb{B}}(C,\tilde{\mathbb{H}})$.

There are a number of questions we can ask about $\mathcal{P}$ for a given $\Pi^1_2$ sentence, but here will only focus on three.   What happens when $\mathcal{P}=\mathcal{ARITH}$,
the set of arithmetic sets? Can $\mathcal{P}$ be the set of all computable sets?
% {\color{red}  (ND:    Is $\mathcal{ARITH}$ standard font?)}
Is $\mathcal{P}$ closed under Turing jump?   In terms of reverse mathematics, a positive answer to the first question implies the sentence is provable in $ACA_0$, a negative answer to the second question implies that $RCA_0$ is not enough to prove the sentence, and a positive answer to the third question implies that  this sentence implies $ACA_0$.

In a forthcoming paper, Cholak, Dobrinen and Towsner will show the first question has a positive answer for all the $\Pi^1_2$ sentences considered here. For node colorings  in $\mathbb{H}_3$ we do provide  a proof in this paper that $\mathcal{ARITH}$ is enough, and more (again see Section~\ref{proof:Indivisibility}).  In some sense this is the base case for the forthcoming proof of Cholak, Dobrinen and Towsner. By \citet{dauriac2023carlsonsimpsonslemmaapplicationsreverse} it is known that if $\mathcal{P}$ is closed under the $\omega$-jump then the $\Pi^1_2$ statements we consider are true in $\mathcal{P}$.

For the second question and nodes as $\mathbb{A}$, Gill \cite{Gill23} provides a negative answer.  For completeness we include a proof of this result using our methods and the idea of a neighborhood set coloring. 
% {\color{red} (ND:  Do you want to say that it is a different proof with a tighter result?)} and some questions drawing from our work in Section~\ref{nodes}. 

For so-called  ``close" pairs in $\mathbb{H}_3$ (as $\mathbb{B}$) and (some) anticliques as $\mathbb{A}$ we provide a positive answer to the third question, i.e.\ any $\mathcal{P}$ which realizes $\forall C \exists \tilde{\mathbb{H}}\mathcal{R}_{\mathbb{A},\mathbb{B}}(C,\tilde{\mathbb{H}})$ must be closed under the Turing jump.  For details, see Theorems~\ref{vcp0} and \ref{CodesJump}, respectively. We need to discuss the framework on these proofs.

\begin{definition}
    Let $\varphi = \forall C \exists H \mathcal{R}(C,H)$ be a $\Pi^1_2$ sentence.  We say that $\varphi$ \emph{codes} $\emptyset^{(k)}$ iff, for all $X$, there is a $C \leq_T X$ such that for $H$ if $\mathcal{R}(C,H)$ then $H$ computes $X^{(k)}$.
\end{definition}

% {\color{red}  (ND:  Is the $C$ above assumed to be computable in $X$?  If yes,  I think we should say that.)}

 Our results here follow an earlier result of Jockusch \cite{J72}. In our current language, Jockusch's result is stated as ``Ramsey's  Theorem for $n$-tuples" (for all finite colorings of $n$-tuples of natural numbers there is a homogeneous set) codes $\emptyset^{(n-2)}$, see Corollary~\ref{joc72}. But here we do not have  homogeneous sets available.  We have to use sets of monochromatic largely $2$-extendable tuples, see Section~\ref{sec:l2e}. Section~\ref{sec:colorings} introduces the needed colorings and shows that the  Jockusch result holds more generally for  monochromatic largely $2$-extendable sets rather than just homogeneous sets.  Here we prove this for $X = \emptyset$ and then use relativization for arbitrary $X$.

A major difficulty in working with the Ramsey property for pairs in $\mathbb{H}_{3}$ or anticliques in $\mathbb{H}_{n+1}$ is how to computably in $\tilde{\mathbb{H}}$ extract the needed monochromatic largely $2$-extendable sets from the resulting $\tilde{\mathbb{H}}$.   This takes some local results about $\mathbb{H}_{n+1}$.  For more details, see Theorems~\ref{le4}, \ref{lpp} and \ref{thm.closele} and the two paragraphs above Theorem~\ref{thm.closele}. Throughout the paper we will use finite structures called
\emph{far anticliques} and \emph{close graphs}, see Definitions~\ref{farac} and \ref{defclose}.  We point out that these finite structure are definable from the given enumeration of main structure, $\mathbb{H}_{n+1}$. 

One interesting corollary of our work on anticliques is that we learn that one of the $\Pi^1_2$ sentences considered here is stronger than ``Ramsey's Theorem for $k$-tuples", see Theorem~\ref{ramsey}.  It turns out that far anticliques of size $k$ have the Ramsey Property in $\mathbb{H}_{n+1}$ (a special case of Theorem 9.9 in \cite{DobrinenJML23}).
Let $C$ be a finite coloring of $k$-tuples of the natural numbers.  This computably and uniformly induces a coloring of far anticliques of size $k$.  Let  $\tilde{\mathbb{H}}$ be a  copy of $\mathbb{H}_{n+1}$ where only one color from the induced coloring appears.  Then from $\tilde{\mathbb{H}}$ we can uniformly compute a homogeneous set for our coloring $C$. 
%{\color{red}  (ND:  Wouldn't you rather talk here about the Jockusch coloring on $2k$-tuples and the stronger result that we get?)}

Now  to our ultimate result: Let $G$ be a finite $(n + 1)$-clique free graph.  Unfortunately, $G$  does not have the Ramsey Property in $\mathbb{H}_{n+1}$, whenever $G$ has at least two nodes.  But, as proved in \cite{DobrinenJML20} and \cite{DobrinenJML23}, for some $k$, the  following $\Pi^1_2$ sentence,  ``$G$ has big Ramsey degree $k$  in $\mathbb{H}_{n+1}$," is true. I.e., there is a computable function $k({G,n})$ such that for all $G$ and $n$ the following statement is true: For all finite colorings $C$ of copies of $G$ in $\mathbb{H}_{n+1}$ there is an $\tilde{\mathbb{H}}$ where at most $k({G,n})$ colors appear among the copies of $G$ in $\tilde{\mathbb{H}}$.  Moreover, as shown in \cite{Balko7}, for some coloring and all corresponding  $\tilde{\mathbb{H}}$, all $k({G,n})$ colors are needed.
We prove that the $\Pi^1_2$ sentence ``$G$ has finite big Ramsey degree $k(G,n)$ in $\mathbb{H}_{n+1}$'' codes $\emptyset^{(|G|-1)}$. A summary of the details follows in the next paragraph but more details, related theorems, and references can be found in Section~\ref{finite}. 

Essentially, for each $G$, there are $k(G,n)$ many computable, definable (in the language of $\mathbb{H}_{n+1}$ and its enumeration) extensions of $G$. In the literature, these extensions are called \emph{diaries}, denoted $\Delta$. All these diaries have the Ramsey Property in $\mathbb{H}_{n+1}$. Let $\Delta_i$, for $i<k(G,n)$, be the diaries for $G$.  Assume we are given a finite coloring $C$ of all the copies of $G$ in $\mathbb{H}_{n+1}$. 
We may first restrict to a copy $\widehat{\mathbb{H}}$ of $\mathbb{H}_{n+1}$, computably defined in $\mathbb{H}_{n+1}$, so that 
each copy of $G$  in  $\widehat{\mathbb{H}}$ computably extends to a unique
extension,
$\Delta_i$.
Color that copy of $G$ by the pair $(i,C(G))$ (this is a well-defined coloring). 
Apply the  Ramsey Properties of the dairies to get an $\tilde{\mathbb{H}}$ (in $\widehat{\mathbb{H}}$)  which is  monochromatic on each dairy of $G$.  
In $\tilde{\mathbb{H}}$, the copies of $G$ have at most $k(G,n)$ colors and hence, $G$ has finite big Ramsey degree $k(G,n)$ in $\mathbb{H}_{n+1}$. From a diary $\Delta$ and $\tilde{\mathbb{H}}$ we  know how to extract monochromatic largely $2$-extendable sets for that diary but unfortunately we do not know how to computably extract such a set. But if we ignore the diaries and make some adjustments on which copies of $G$ (just the close ones) we will consider, we can computably in $\tilde{\mathbb{H}}$ extract a  largely $2$-extendable set of indexes for copies of $G$ in $\tilde{\mathbb{H}}$. If we let $C = J_{|G|+1}$ be the Jockusch coloring and we ignore the first 
coordinate
%comportment 
in our coloring, we can show this largely $2$-extendable set extracted computably is monochromatic. And now we can apply our methods from above.  This shows that the sentence ``$\forall C \exists \tilde{\mathbb{H}} \bigwedge_{i<k} \mathcal{R}_{\Delta_i,\mathbb{H}_{n+1}}(C,\tilde{\mathbb{H}})$" codes $\emptyset^{|G|-1}$.  
It turns out that in our domain of discourse, the two sentences ``$\forall C \exists \tilde{\mathbb{H}} \bigwedge_{i<k} \mathcal{R}_{\Delta_i,\mathbb{H}_{n+1}}(C,\tilde{\mathbb{H}})$" and ``$G$ has finite big Ramsey degree  $k(G,n)$ in $\mathbb{H}_{n+1}$'' are (strongly Weihrauch) equivalent. For more details, see Theorem~\ref{thm.onecolornearG} and the following two corollaries.

\section{Largely $2$-Extendable}\label{sec:l2e}

 Fix $X \subseteq \mathbb{N}$ infinite. $X$ is linearly ordered by $<$.  In this section we will use $n$ to denote the size of unordered tuples from $\mathbb{N}$.  Recall that $[X]^n = \{ Y \subseteq X : | Y | = n \}$, the set of unordered tuples of size $n$. We will denote such a tuple as $\bar{x} = ( x_0, x_1, \ldots, x_{n-1} )$, where $x_0 < x_1< \ldots < x_{n-1}$. For $n=0$, the only tuple considered  is the empty set, $\emptyset$. Let ${\mathcal{F}}$ be a family of $n$-tuples and let ${\mathcal{F}}_k$ be the restriction of the tuples in ${\mathcal{F}}$ to $k$-tuples, for $k \leq n$. So if $(  x_0, x_1, \ldots, x_{n-1} ) \in {\mathcal{F}}$ then $(  x_0, x_1, \ldots, x_{k-1} ) \in {\mathcal{F}}_k$.

\begin{definition}\label{def.l2e}
	We say a  family of tuples, $\mathcal{F} \subseteq [X]^n$, is \emph{largely $2$-extendable} iff 
    \begin{enumerate}
        \item $\mathcal{F}_1$ is infinite, 
        \item there exists an $x_0$ with infinitely many $x_1$ where $(x_0,x_1) \in \mathcal{F}_2$, and 
        \item for all $ (  \tilde{x}_0, \tilde{x}_1, \ldots, \tilde{x}_{n-1} ) \in \mathcal{F}$ (equivalently we can start with $ (  \tilde{x}_0, \tilde{x}_1) \in \mathcal{F}_2$), the following holds: 
		\begin{equation}\label{eq.LE}
	\begin{split}
		\forall l_2 \exists x_2 > l_2 \forall l_{3} \exists &x_{3} > l_{3} \ldots \forall l_{n-1} \exists x_{n-1} > l_{n-1} \\ (\tilde{x}_0, \tilde{x}_1, & x_2, x_{3}, \ldots x_{n-1} )  \in \mathcal{F}.
        	\end{split}
	\end{equation}
       \end{enumerate}
	
\end{definition}

Both $[\mathbb{N}]^n$ and $[X]^n$, for $X\sse\mathbb{N}$ infinite,  are largely $2$-extendable, for all $n$. If we were to Skolemize the definition (i.e., add functions $S_j$ which output $x_j$ from the previous data including $l_j$), there is no requirement that such functions should be computable. In our above two examples, however, the Skolem functions are computable (or as least computable in $X$): look for the least $x_j$ $(\in X)$ greater than $l_j$.  We will see some examples later where the Skolem functions are not computable.  

Let $i$ be such that $0\leq i <n $. An $i$-tuple, $ (x_0,x_1, \ldots x_{i-1}) \in {\mathcal{F}}_i$, has \emph{large $1$-point extensions} w.r.t.\ $\mathcal{F}$ iff for all $l_{i}$, there is an $(i+1)$-tuple, $(x_0,x_1, \ldots x_{i-1}, x_i) \in {\mathcal{F}}_{i+1}$, where $x_i > l_i$.  A tuple which does  not have large $1$-point extensions is called  a \emph{stub}.  The goal is to avoid stubs,  particularly in $\mathcal{F}_2$. 

We will use the following lemma to show a family $\mathcal{F}$ is largely $2$-extendable.  

\begin{lemma}\label{le1} Assume that $\mathcal{F}_2$ is largely $2$-extendable. Assume that there exists a subfamily of tuples, $\tilde{\mathcal{F}} \subseteq \mathcal{F}$, such that $\mathcal{F}_2 = \tilde{\mathcal{F}}_2$ and for all $i$, where $2 \leq i < n$, every tuple in $\tilde{\mathcal{F}}_i$ has large $1$-point extensions w.r.t.\ $\tilde{\mathcal{F}}$. Then $\mathcal{F}$ is largely $2$-extendable.  
\end{lemma}

\section{Colorings}\label{sec:colorings}

Classically, the infinite Ramsey theorem  says that for every finite coloring of $[\mathbb{N}]^n$ there is an infinite homogeneous set  $H$, meaning that on $[H]^n$ only one color is used. Jockusch \cite{J72} showed that there is a computable finite coloring of $[\mathbb{N}]^n$ where any homogeneous $H$ set must compute $\emptyset^{(n-2)}$. We will call a largely $2$-extendable set $\mathcal{F}$ {\em monochromatic  w.r.t.\  $C$} if all the tuples in  $\mathcal{F}$ have the same color w.r.t.\ a coloring $C$.  In this section we will concretely and inductively spell out these colorings with the goal that if we have a monochromatic largely $2$-extendable set, then this set also must compute $\emptyset^{(n-2)}$.  

%To make our lives easier we denote the Jockusch coloring  $J(\bar{x})$ as $J(x_0,x_1, \ldots, x_{n-1})$. Also 
If we have a coloring $C$ of $[\mathbb{N}]^l$, where $l < n$, then we can directly extend $C$ to a coloring of $[\mathbb{N}]^n$ by defining $C(x_0,x_1, \ldots, x_{n-1})$  as $C(x_0,x_1, \ldots, x_{l-1})$.
%where $x_0< x_1 <  \ldots < x_{n-1}$. 
Let $\{K^X_s : s \in \omega \}$  be the standard  
stagewise approximations to 
$K^X=X'$, the halting set relative to the oracle $X$.  It is important to note that we are using the standard convention that $W^X_{e,s}$ can only use the oracle below $s$; i.e., $W^X_{e,s} = W^{X \upharpoonright s }_{e,s}$. The limit lemma says that if $f$ is computable in $X'$, then there is an $X$-computable sequence $\langle f_s:s<\omega\rangle$ of $X$-computable functions such that $\lim_s f_s = f$.

\begin{definition}[Jockusch's Colorings]\label{JC} Let $J^{X,3}(x,s,t) =1 $ iff $K^X_s \upharpoonright x = K^X_t \upharpoonright x$ %and, for all $e< x$ we have $$\Phi_{e,s}^{X'_s}(e) = \Phi_{e,t}^{X'_t}(e)$$ (this includes the possibly both sides diverge) 
and $0$ otherwise. This coloring is computable in $X$.  Inductively, for $i\ge 3$, assume $J^{X',i}$ is computable in $X'$. Apply the limit lemma to  $J^{X',i}$ to get the $X$-computable sequence $\langle J^{X,i}_{s}:s<\om\rangle$.  Let  $$
J^{X,i+1}(x_0, x_1, \ldots, x_{i-1}, x_{i}) = J^{X,i}_{x_{i}}(x_0, x_1, \ldots, x_{i-1}).
$$ 
Then by induction, for all $k$, $J^{X,k}$ is $X$-computable. We can use this iterative process to define an $X$-computable coloring $J^{X,3+k}$ from the $X^{(k)}$-computable coloring $J^{X^{(k)},3}$.  For $n\ge 3$, let $J^X_n = \prod_{i=3}^{n} J^{X,i}$. If $X = \emptyset$ we will drop the $X$ and just use $J_n$. 
\end{definition}
	
\begin{theorem}\label{JCLE}
	Color a family $\tilde{\mathcal{F}} \subseteq [\mathbb{N}]^n$ with the coloring $J^X_n$.  Let $\mathcal{F}$ be a monochromatic w.r.t.\ $J^X_n$ largely $2$-extendable subfamily. Then 
	\begin{enumerate}
		\item  $J^X_n$ colors all the tuples in $\mathcal{F}$ color $1$.
		\item If $(  x_0, {x}_1, x_2, \ldots, x_{n-1} ) \in \mathcal{F}$ then for all $i$, where $0 \leq i < n-2$, $K^{X^{(i)}}_{x_1} \upharpoonright x_0 = K^{X^{(i)}} \upharpoonright x_0$.  
		\item $\mathcal{F}$ computes $X^{(n-2)}$.
	\end{enumerate}
\end{theorem}

\begin{proof}  By relativization it is enough to show this for $X =\emptyset$.

	 (1) It is enough to  find a tuple in $\mathcal{F}$ which is color $1$. There is an $x_0 > 0 $ with infinitely many $x_1$ where $ (x_0, x_1) \in \mathcal{F}_2$. Computably in $K^{\emptyset^{(n-3)}}=\emptyset^{(n-2)}$ we can find an $l_1$ such that for all $s > l_1$ and for all $i$, if $0 \leq i \leq n-3$, then $K^{\emptyset^{(i)}}_s \upharpoonright x_0 = K^{\emptyset^{(i)}} \upharpoonright x_0$ (it is important to note we are using the oracle $\emptyset^{(i)}$ itself and what we are approximating is $\Sigma_1$ in $\emptyset^{(i)}$). 
	 	 
	 Find $x_1 > l_1$ such that $(x_0,x_1) \in \mathcal{F}_2$.  Choose $l_2 > x_1$ and use $l_2$ and the fact that $\mathcal{F}$ is largely $2$-extendable to find $x_2$ where $ (x_0, x_1, x_2) \in \mathcal{F}_3$. Then for all $i$, where $0 \leq i < n-2$, $J^{\emptyset^{(i)},3}(x_0,x_1,x_2)=1$.  	
	 
	 Inductively assume we have $(x_0,x_1,x_2, \ldots, x_{j+2}) \in \mathcal{F}_{j+3}$ such that for all $i$, if  $0 \leq i < n-(j+2)$, then $J^{\emptyset^{(i)},j+3}(x_0,x_1,x_2, \ldots, x_{j+2})=1$.  Now there is an $l_{j+3}$ such that for all $s > l_{j+3}$, for all $i$ such that $0 < i < n-(j+2)$, we have
     $$
     J^{\emptyset^{(i-1)},j+3}_{s}(x_0, x_1, \ldots, x_{j+2}) = J^{\emptyset^{(i)},j+3}(x_0, x_1, \ldots, x_{j+2})=1.
     $$
     Use $l_{j+3}$ to find $x_{j+3}$ such that $(x_0,x_1,x_2, \ldots, x_{j+2},x_{j+3}) \in \mathcal{F}_{j+4}$. 	 We have that for all $i$ such that $0 < i < n-(j+2)$,  $$J^{\emptyset^{(i-1)},j+4}(x_0,x_1,x_2, \ldots, x_{j+2}, x_{j+3})=1.$$ 	
	 In the end, this implies that $J_n(x_0,x_1, \ldots, x_{n-1})=1$.  
	 
	 (2) Assume otherwise. Then 
     there is some $\{x_0,x_1,\tilde{x}_2,\dots,\tilde{x}_{n-1}\}\in\mathcal{F}$ 
     such that $K^{\emptyset^{(i)}}_{x_1} \upharpoonright x_0 \neq K^{\emptyset^{(i)}} \upharpoonright x_0$, for some $0<i<n-2$.  There is an $l_2$ such that for all $s > l_2$, $K^{\emptyset^{(i)}}_{x_1} \upharpoonright s \neq K^{\emptyset^{(i)}}_s \upharpoonright x_0$.  Use the property (\ref{eq.LE}) to get an $x_2$ such that $(x_0,x_1,x_2) \in \mathcal{F}_3$. Then we have that $J^{\emptyset^{(i)},3}(x_0,x_1,x_2)=0$. 
	  
	  We want to show that the coloring $J^{\emptyset,i+3}(x_0,x_1,x_2,\ldots, x_{i+2})=0$.  If $i = 0$ we are done.  Inductively assume we have $(x_0,x_1,x_2, \ldots, x_{j+2}) \in \mathcal{F}_{j+3}$ such that the coloring $J^{\emptyset^{(i-j)},j+3}(x_0,x_1,x_2, \ldots, x_{j+2})=0$. Then, by the limit lemma, there is an $l_{j+3}$ such that for all $s > l_{j+3}$, $$J^{\emptyset^{(i-(j+1))},j+3}_{s}(x_0, x_1, \ldots, x_{j+2}) = J^{\emptyset^{(i-j)},j+3}(x_0, x_1, \ldots, x_{j+2})=0.$$ Use $l_{j+3}$ to find $x_{j+3}$ such that $(x_0,x_1,x_2, \ldots, x_{j+2},x_{j+3}) \in \mathcal{F}_{j+4}$.  We have that 
     $$
     J^{\emptyset^{(i-(j+1))},j+4}(x_0,x_1,x_2, \ldots, x_{j+2}, x_{j+3})=0.
     $$ 	
     This induction implies that $J^{\emptyset,i+3}(x_0,x_1,x_2, \ldots, x_{i+2})=0$ and hence the product coloring $J_n(x_0,x_1, \ldots, x_{n-1})=0$, a contradiction to (1).  
	 
	  (3) We want to compute $\emptyset^{(n-2)} \upharpoonright l_0$ just using membership in $\mathcal{F}$.  Here we will be using $n$ times the fact that $\mathcal{F}_1$ is infinite. We will start by letting $\{ x^{n-3}_0, x^{n-3}_1, x^{n-3}_2, \ldots, x^{n-3}_{n-1} \}$ be the first tuple enumerated into $\mathcal{F}$, where $x_0 > l_0$.   Given $\{ x^{j}_0, x^{j}_1, x^{j}_2, \ldots, x^{j}_{n-1} \}$, we let  $\{ x^{j-1}_0, x^{j-1}_1, x^{j-1}_2, \ldots, x^{j-1}_{n-1} \}$ denote  the first tuple $\{ x_0, x_1, x_2, \ldots, x_{n-1} \}$  enumerated in $\mathcal{F}$, where $x_0 > x^j_{n-1}$.  
	  
	  By (2), we have that $K^{\emptyset^{(j)}}_{x^j_1} \upharpoonright x^j_0 = K^{\emptyset^{(j)}} \upharpoonright x^j_0$. Hence, computably in $\mathcal{F}$, we know $ \emptyset' \upharpoonright x^0_0 = K^{\emptyset^{(0)}} \upharpoonright x^0_0 =K^{\emptyset^{(0)}}_{x^0_1} \upharpoonright x^0_0$.   Inductively assume we can compute  ${\emptyset^{(j)}} \upharpoonright x^{j-1}_0$. We have that $${\emptyset^{(j+1)}} \upharpoonright x^{j}_0 = K^{\emptyset^{(j)}} \upharpoonright x^{j}_0 =K^{\emptyset^{(j)}}_{x^{j}_1} \upharpoonright x^{j}_0 = K^{\emptyset^{(j)}\upharpoonright x^{j}_1}_{x^{j}_1} \upharpoonright x^{j}_0 = K^{\emptyset^{(j)}\upharpoonright x^{j-1}_0}_{x^{j}_1} \upharpoonright x^{j}_0.$$ The last equality holds since $x^{j-1}_0 > x^j_{n-1} > x^j_1$. So, computably in $\mathcal{F}$, we know $ \emptyset^{(n-2)} \upharpoonright x^{n-3}_0$ and $x^{(n-3)}_0 > l_0$.  
\end{proof}

\begin{corollary}[Jockusch]\label{joc72}
	If we color $[\mathbb{N}]^n$ by $J^X_n$ then a homogeneous set computes $X^{(n-2)}$. So ``Ramsey Theory for $n$-tuples'' codes $\emptyset^{(n-2)}$. 
\end{corollary}

\subsection{A failed improvement}

For a homogeneous set $H$ w.r.t.\ a coloring $C$ of $[\mathbb{N}]^n$, the set $[H]^n$ is largely $2$-extendable. What about the converse?   There are colorings of $[\mathbb{N}]^n$ without any computable homogeneous sets (for $n=2$, see \cite{MR278941}, and for $n>2$, see the above Jockusch colorings). For $n=2$, by chasing the definition of largely $2$-extendable, we can see that every finite coloring of $[\mathbb{N}]^n$ has (nonuniformly) a computable monochromatic largely $2$-extendable set. Let $n>2$. Gavin Dooley showed, for any finite coloring $C$ of $[\mathbb{N}]^n$ it is possible to find a new coloring $\tilde{C}$ where $\tilde{C}$ has computable monochromatic largely $2$-extendable sets but any homogeneous set w.r.t.\ $\tilde{C}$ computes a homogeneous set w.r.t.\ ${C}$.  Consider a finite coloring $C$ of $[\mathbb{N}]^n$ without any computable monochromatic largely $2$-extendable sets. It is open if a monochromatic largely $2$-extendable set w.r.t.\ $C$ must always compute a homogeneous set  w.r.t.\ $C$.  We make this question very concrete for a particular coloring below in Question~\ref{cc3}.

Now work with $n=3$.  Hirschfeldt and Jockusch \cite{MR3518779} showed there is a finite coloring of $[\mathbb{N}]^3$ such that every homogeneous set is PA over $0'$. The proof of this extension fails if  we consider largely $2$-extendable sets.  

It is worth going into the details of the Hirschfeldt and Jockusch result to partially understand the difference between monochromatic largely $2$-extendable sets and homogeneous sets. For this we need to assume our functionals are $2$-valued. 

\begin{definition}[Hirschfeldt and Jockusch Colorings]
  We use the coloring $J_3$ from Definition~\ref{JC}. We say that $H_3(x,s,t) =1 $ iff for all $e< x$, we have $$\Phi_{e,s}^{X'_s}(e) = \Phi_{e,t}^{X'_t}(e)$$ (this includes the possibility that  both sides diverge) 
and $0$ otherwise. 
\end{definition}

\begin{corollary}[Hirschfeldt and Jockusch]
	If we color $[\mathbb{N}]^3$ by $J_3\times H_3$ then a homogeneous set $H$ is PA over $\emptyset'$. 
\end{corollary}

\begin{proof}
	By Theorem~\ref{JCLE} we know for all tuples $\vec{x} \in [H]^3$, $J_3(\vec{x}) =1 $. So it is enough to show, (I) there is a tuple, $\vec{x}$, where $H_3(\vec{x})=1$ and then use that to show (II) $H$ computes a completion of $\Phi_e^{\emptyset'}(e)$. 
	
	(I) The first part mirrors the first two paragraphs of the proof of (1) of Theorem~\ref{JCLE}.  Pick $x_0 \in H$.  Then there is an $l$ such that, for all stages $s >l$ and for all $e<x_0$, if $\Phi_{e}^{X'}(e)\converge$ then $\Phi_{e,s}^{X'_s}(e) =\Phi_{e}^{X'}(e)$. 
	
	This next part is very different.  Since there are only finitely many partial functions from $x_0$ to $\{0,1,\uparrow\}$, we can find an infinite set $Z$ of $x_1>s$ in $ H$ such that 
     the value of the partial function $\Phi_{e,{x_1}}^{X'_{x_1}}(e)$ agrees for all $x_1\in Z$. Let $x_1, x_2$ be in $Z$ such that $x_1 < x_2$.  Then $(x_0,x_1,x_2) \in [H]^3$ and $H_3(x_0,x_1,x_2)=1$.  
	
	(II)  We compute our claimed completion $f$ as follows: for $e$ and $x_0,x_1 \in H$ where $e < x_0$ let $f(e) = \Phi_{e,{x_1}}^{X'_{x_1}}(e)$.  By (I), $f$ is well defined.   We claim that if $\Phi_{e}^{X'}(e)\converge$ then $\Phi_{e}^{X'}(e)=f(e)$. If not, then there is an $l$ such that, for all stages $s >l$, $\Phi_{e,s}^{X'_s}(e) \neq \Phi_{e}^{X'}(e)$. Pick ${x}_2$ such that ${x}_2 \in H$ and ${x}_2 > s$. Then $H_3(x_0,x_1,x_2)=0$. Contradiction.	
\end{proof}

Now let's consider using a monochromatic largely $2$-extendable set $\mathcal{F}$ rather than a homogeneous set $H$ for the proof of (I).  First choose $x_0$ as done in the proof of Theorem~\ref{JCLE} part (1). The set $Z$ in the second part of the proof of (I) exists, even if we refine it such that $(x_0,x_1) \in \mathcal{F}_2$.   But there is no reason to believe for all (any) pairs $x_1,x_2 \in Z$, the tuple $(x_0,x_1,x_2)$ is in $\mathcal{F}_3$.  Later we will see an example where for all pairs $x_0,x_1 \in Z$, the sets $\{ x_3 | (x_0,x_1,x_3) \in \mathcal{F}_3 \}$ are pairwise disjoint. We are left with the following questions:

\begin{question}\label{cc3}
	If we color $[\mathbb{N}]^3$ by $J_3\times H_3$ then does every monochromatic largely $2$-extendable set $\mathcal{F}$ compute a homogeneous set or even just completion of $\Phi_e^{\emptyset'}(e)$?
\end{question}

Jockusch \cite{J72} has a large number of results about coloring of tuples and the resulting homogeneous sets.  What can be said about monochromatic largely $2$-extendable sets rather than homogeneous sets?  Also we can take those colorings into different settings and see what happens.  We will do this partially below.  There are also some possible extensions to largely $k$-extendable.   We are are working to complete this task.

\section{An introduction to $\mathbb{H}_{n+1}$} \label{H3}

We will be working with countable graphs $(V,E)$.  $K_{n}$ is the clique with $n$ vertices or nodes, so  there is an edge between every pair of vertices;  we also say that all pairs of vertices in $K_n$ are connected. $\overline{K}_{n}$ is the anti-clique with $n$ nodes, so every pair of nodes is not connected. 
In particular, 
$K_2$ is a pair of connected nodes, $\overline{K}_2$ is a pair of disconnected nodes, and $K_3$ is a triangle, 3 nodes with all pairs connected.  A $K_{n+1}$-free graph is a graph which does not contain a copy of $K_{n+1}$.  Our focus will be the countable $K_{n+1}$-free graphs $\mathbb{H}_{n+1}$, $n\ge 2$,  with the following extension property:

\begin{definition}[Extension Property]
A $K_{n+1}$-free graph $G= (V,E)$ has the \emph{extension property} if for all finite disjoint subsets of nodes $D$ and $F$ of $V$, if there is no copy of $K_{n}$ in the subgraph induced by $E$ on $D$ then there is a node $v$ in $V$ such that for each $\tilde{v} \in D$, $v$ and $\tilde{v}$ are connected (we say $v$ and $D$ are connected) and  for each $\tilde{v} \in F$, $v$ and $\tilde{v}$ are disconnected (we say $v$ and $F$ are disconnected or not connected).
\end{definition}

 The extension property implies that $\mathbb{H}_{n+1}$ is not finite.  Moreover, it is a defining property of $\mathbb{H}_{n+1}$, for the extension property holds in a $K_{n+1}$-free graph if and only if that graph is the \Fraisse\ limit of the \Fraisse\ class of all finite $K_{n+1}$-free graphs.
 
 It is a straightforward construction to build a computable copy of $\mathbb{H}_{n+1}$.  By the standard back and forth argument any pair of (computable) copies $\mathbb{H}_{n+1}$ are (computably) isomorphic.   Since $\mathbb{H}_{n+1}$ is countable we can fix an enumeration of $\mathbb{H}_{n+1}$.    So $\mathbb{H}_{n+1} = (\{\tilde{v}_i | i \in \omega\}, E)$. We will always be working inside this graph and will assume that our edges are given by $E$.

 Now we will want an enumeration, $\{\hat{v}_i | i \in \omega\}$,  with the special property that  $\hat{v}_{i+1}$ and $\hat{v}_{i}$ are connected, for each $i\in \omega$.
We  re-order the vertices as follows.
 Pick some $\hat{v}_0$.   At stage $s+1$, let $i$ be the least such that $\tilde{v}_i$ does not  appear among the nodes, $\hat{v}_0, \hat{v}_1, \ldots ,\hat{v}_s$,
 and $\tilde{v}_i$  is  connected to $\hat{v}_s$; let $\hat{v}_{s+1}=\tilde{v}_i$.
 Such a node exists, by the extension property. 
Since each node $\tilde{v}_i$ is connected to infinitely many nodes in $\mathbb{H}_{n+1}$, 
each $\tilde{v}_i$ eventually appears as some $\hat{v}_s$.
% If $\tilde{v}_i$ and $\hat{v}_s$ are connected let $\hat{v}_{s+1} = \tilde{v}_i$. Otherwise, via the extension property, there is a node $\tilde{v}$ such that $\tilde{v}$ and $\tilde{v}_i$ are connected, $\tilde{v}$ and $\hat{v}_{s}$ are connected, and $\tilde{v}$ is not connected to any of $\hat{v}_0, \hat{v}_1, \ldots ,\hat{v}_{s-1}$. Let $\hat{v}_{s+1} = \tilde{v}$ and $\hat{v}_{s+2} = \tilde{v}_i$.  \pc{Check above here to end of paragraph} 
%Please note that $\{\hat{v}_i | i \in \omega\}$ is just a reordering of the nodes $\{\tilde{v}_i | i \in \omega\}$. 
So the graphs $(\{\hat{v}_i | i \in \omega\},E)$ and $(\{\tilde{v}_i | i \in \omega\}, E)$ are isomorphic, but the identity map is not necessarily an isomorphism between our two copies of $\mathbb{H}_{n+1}$.

We will set $\mathbb{H}_{n+1} = (\{\hat{v}_i | i \in \omega\}, E)$ and call this our {\em standard enumeration} or {\em standard copy}. (We also have no use for the first enumeration $\{\tilde{v}_i | i \in \omega\}$ and will reuse the tildes.)
 %The satisfaction of the first bullet point implies that the extension property holds for $\mathbb{H}_{n+1}$.
 From now on we will refer to nodes (sets of nodes) by their index (sets of indices) in the standard enumeration.   From now on when we construct a copy of $\mathbb{H}_{n+1}$, our convention will be that we are working within a subcopy of $\mathbb{H}_{n+1}$ contained in our standard copy. So the edges are induced by $E$.

\subsection{The Greedy Algorithm}\label{greedy}
 
 When we construct subcopies of $\mathbb{H}_{n+1}$, we will also stagewise copy our standard enumeration $\{ \hat{v}_i | i \in \omega \}$.  So we will select nodes ${v}_i$ to ensure that at stage $s$ the nodes $\hat{v}_0, \ldots, \hat{v}_s$ and $v_0, \ldots, v_s$ are isomorphic as finite graphs (with edges given by $E$). The greedy computable stagewise construction of a subcopy of $\mathbb{H}_{n+1}$ goes as follows: at stage $s+1$ select a node ${v}_{s+1}$ such that, for each $i\le s$, ${v}_{s+1}$ is connected to ${v}_i$ iff ${\hat{v}}_{s+1}$ is connected to ${\hat{v}}_i$.  Also we will always want the index for ${v}_{s+1}$ to be larger than any of the indices for earlier enumerated nodes. Sometimes we will also add a threshold $t_s$ which  ${v}_{s+1}$ must be larger than.  If we do this construction within another subcopy  of $\mathbb{H}_{n+1}$, finding such a node is not a problem.  But sometimes we will add some additional properties on the selection of $v_{s+1}$ which can be more problematic.

 The greedy algorithm also applies  to building copies of any finite $G = \{ {v_i | i < g}\}$ which is $K_{n+1}$ free within a copy of $\mathbb{H}_{n+1}$.  Two enumerations, $\{ {v_i | i < g}\}$ and $\{ {\tilde{v}_i | i < g}\}$, of $G$ are {\em order-isomorphic} iff the
 function $v_i\mapsto \tilde{v}_i$ 
 induces an automorphism of $G$. Many graphs $G$ have many non-order-isomorphic enumerations. When the greedy algorithm is applied  to an enumeration of $G$, the resulting enumeration (from the greedy algorithm) is order-isomorphic to the original.

 \subsection{Neighbor Sets}\label{Neighbor}

  Two nodes are {\em neighbors} iff they are connected.    
 \begin{definition}
   The \emph{neighborhood set} of $\hat{v}_n$ (or $m$), $N(\hat{v}_n)=N_n$,  is the set of nodes $v$ such that $v$ is  connected to $\hat{v}_n$ and, for all $k < n$, $v$ is not connected to  $\hat{v}_k$. 
 \end{definition}

(We will abuse notation and use $n$ for indexing of neighbor sets within a fixed $\mathbb{H}_{n+1}$.)  So our neighborhood sets are disjoint. Inside $\mathbb{H}_{3}$,  two neighbors of a given node cannot be connected. Inside $\mathbb{H}_{n+1}$, a neighborhood  set cannot contain a copy of $K_n$.  %But, by the above remark, it does contain a copy of $\mathbb{H}_n$ and later we will consider neighborhood sets within that copy.   
 Each node $\hat{v}_j$ is in  some neighborhood set $N_n$ where $n<j$ (for $j=0$ we let $n=0$).  At times, we will label our nodes as $\hat{v}_j = \hat{v}_{n,j}$, where $\hat{v}_j \in N_n$. This labeling is a structural property induced by our standard enumeration.

\section{Indivisibility and Properties of subcopies of $\mathbb{H}_{n+1}$} \label{indi}

From now on, we will work with a subcopy $\tilde{\mathbb{H}}$ within our standard copy of $\mathbb{H}_{n+1}$.  The indexing and neighborhood sets are determined by our standard copy.  %We need the following theorems. 

%{\color{red}  (ND:  I should go back and reference the ``standard enumeration" for $n\ge 3$ given in the Appendix.)}

\subsection{Indivisibility} % and the Mycielski Graph

The following theorem is due to Komjath and R\"{o}dl \cite{Komjath/Rodl86} for $n=2$, and due to
 El-Zahar and Sauer \cite{El-Zahar/Sauer89} for all $n\ge 3$. Indivisibility means that nodes have the Ramsey Property in $\mathbb{H}_{n+1}$.

\begin{theorem}[Indivisibility]\label{Indivisibility}
	For each  $n\ge 2$, for  each finite $k$ and every $k$-coloring of the nodes in $\mathbb{H}_{n+1}$ there is a homogeneous copy of $\mathbb{H}_{n+1}$ (a copy where all the nodes have the same color).
\end{theorem}

 Indivisibility will be a useful tool. We will provide a proof for $n=2$ in Section~\ref{proof:Indivisibility}. 

% \begin{theorem}[Mycielski Graph] For all $l$, there is a finite triangle free graph $G_l$ with chromatic number $l+1$.  Moreover $G_l$ can be found effectively and uniformly from $l$. 
% \end{theorem}

% The Mycielski Graph can be found in standard graph theory textbooks.  By the extension property we can always find a copy of $G_l$ within $\tilde{\mathbb{H}}$.  We will use the Mycielski Graph to prove a more effective and local version of indivisibility.  

\subsection{A Critical Theorem}

The following theorem will prove very useful.  
While these facts have been known for a long time, 
this is the first time we are aware of where they are  isolated in a theorem.

\begin{theorem}[Folklore]\label{local1}
\begin{enumerate} 
    \item $\tilde{\mathbb{H}}$
    is not contained
     in  finitely many neighborhood sets.
	\item For each $l_0$, we can always find a node $v_{n,j}$ in $\tilde{\mathbb{H}}$, where $l_0 < n < j $.
	\item 
    For each $l_0$, there is an $n > l_0$ so that  for all $l_1$, there is $j > l_1$ where the node $v_{n,j}$ is in $\tilde{\mathbb{H}}$.
	\item For each $l_0$, there is a node $v_n \in \tilde{\mathbb{H}}$ and infinitely many nodes $v_{n_i,j_i} \in \tilde{\mathbb{H}}$, where $l_0 < n_i \leq n$,  such that $v_{n_i,j_i}$ and $v_n$ are connected. 
\end{enumerate}
\end{theorem}
	
(2) Follows directly from (1).  (3) and (4) do not. Note that (2), (3), and (4)  each imply (1).

\begin{proof}  (1)  Assume $\tilde{\mathbb{H}}$ lives in $l_0$ neighborhood sets. $l_0$-color the nodes in $\tilde{\mathbb{H}}$ by membership in the various neighborhood sets. By indivisibility there is a copy of  $\mathbb{H}_{n+1}$ living in some neighborhood set.  But no neighborhood set can contain a copy of $K_n$. Contradiction.

  (4): $2$-color $\mathbb{H}_{n+1}$ via $C(v_{m,j}) = BLUE$ iff $m > l_0$. By indivisibility and (1), there must be a blue copy of $\mathbb{H}_{n+1}$.  Let $v_{m}$ be the least node in this blue copy.  There are infinitely many  nodes $v_{n_i,j_i}$ in this blue copy connected to $v_m$. We have for all $i>0$, $l_0 < n_i < m$. 
  
  (3) Hence there must be a least $\tilde{n}$ such that $l_0 < \tilde{n} \leq n$ and there are infinitely many $i$ where the nodes $v_{\tilde{n},j_i}$ are in our blue copy and hence the original copy. 
  \end{proof}

(2) is effective. Since we know such a node always exists we can search for it. But in (3) we do not find $n$ effectively, see Lemma~\ref{notcomputable}. 
We call the coloring used above a \emph{neighborhood set coloring}, as the color of a node solely depends on the node's neighborhood set.

It will be useful to consider the following one point extension theorems.  We often will apply them without mention below.  They use parts of Theorem~\ref{local1} and the extension property. The first uses Part (3) of Theorem~\ref{local1} and is an upwards extension lemma. The second is a downwards extension lemma. It uses the fact that
for each $n$ there is a copy of $\tilde{\mathbb{H}}$ in which  all nodes
are not connected to $v_n$ or to any nodes in the first $n+1$ neighborhoods sets of $\tilde{\mathbb{H}}$.

\begin{lemma}\label{build}
	For all finite graphs $H$ in $\tilde{\mathbb{H}}$ and for all $l$, there is an $n>l$ and infinitely many $j$, such that $H \sqcup \{v_{n,j} \}$ is a one point extension of $H$, there are no edges between any node in $H$ and $v_{n,j}$, and  $n$ is greater than any of the indices in $H$. 
\end{lemma}

\begin{lemma} \label{extdown}
   Fix $\tilde{\mathbb{H}}$. Let $n$ be given by Lemma~\ref{local1} (4). Let $H_0$ be a finite subgraph of $\tilde{\mathbb{H}}$ where none of $H_0$'s nodes are in the first $n+1$ neighborhood sets.  Let $H_1$ be any one point extension of $H_0$.   Then there are infinitely many nodes $v_{n_i,j_i}$ where $H_0 \cup \{ v_{n_i,j_i} \}$ is graph isomorphic to $H_1$, $v_n$ and $v_{n_i,j_i}$ are connected, and $n_i \leq n$.  
\end{lemma}

\subsection{A failed extension to Theorem~\ref{local1} (3)}

\begin{lemma}\label{notcomputable}
	There is a computable copy $\tilde{H}$ with no computable function $g$ such that for all $l$, $l < g(l)$ and $n=g(l)$ satisfies the conclusion of Theorem~\ref{local1} (3). 
\end{lemma}

\begin{proof}
	It is enough to stagewise use a slightly modified version of the greedy algorithm (see Subsection~\ref{greedy}) to build a copy of $\mathbb{H}_{n+1}$, $\tilde{\mathbb{H}}$, to meet all the following requirements.  $\mathcal{R}_e$: There is $l_e$ such that if $\varphi_e(l_e)\converge =n_e$ then in $\tilde{\mathbb{H}}$ there are only finitely many nodes of the form $v_{n_e,j}$.  This will be done by finite injury argument. When we are allowed to (re)start working on $\mathcal{R}_e$, choose an $l_e$ large and reset all $\mathcal{R}_{e'}$, for $e' > e$.  As the greedy algorithm adds nodes to $\tilde{\mathbb{H}}$ the default action will be to have them disconnected from $v_{l_e}$. Let $s$ be the first stage where $\varphi_{e,s}(l_e)\converge =n_e> l_e$.  Again the default action is to only add nodes disconnected from $v_{n_e}$ to $\tilde{\mathbb{H}}$, when possible.  This meets the requirement if $v_{n_e}$ is not in $\tilde{\mathbb{H}}$ by stage $s$.  If $v_{n_e}$ is in $\tilde{\mathbb{H}}$ by stage $s$ then, by the default action, $v_{l_e}$ and $v_{n_e}$ are not connected. It is OK to add a node connected to $v_{n_e}$ if it also must be connected to some other node in $\tilde{H}$ of lower index (determined by the greedy algorithm).  If the greedy algorithm wants to add a node connected to $v_{n_e}$ but not connected any other node in $\tilde{H}$ of lower index, we will ask that node also be connected to $v_{l_e}$.  Since $v_{l_e}$ and $v_{n_e}$ are not connected, such a node can always be found. 
\end{proof}

\section{Proof of Indivisibility for $n+1=3$}\label{proof:Indivisibility}

 This section will focus on a proof of Indivisibility for $n+1=3$.  This section can be read independently from the remaining sections.

\begin{proof}[Proof for $n=2$]
	Our first step is to have the coloring induce a $2$-coloring on the neighborhood sets as follows: 

%\begin{remark}[The extension $2$-coloring on the neighborhood sets]
Assume our colors are RED and BLUE.  Color a neighborhood set, $N_i$, RED iff, for all finite disjoint subsets of nodes $D$ and $F$ of $V$ such that 
 \begin{itemize}
  \item $i \in D$,
 	 \item for all $l<i$, $l \notin D$, 
 	\item there are no connected nodes inside $D$
 	
 \end{itemize} then there is a RED node $v$ such that for all $\tilde{v} \in D$, $v$ and $\tilde{v}$ are connected and, for all  $\tilde{v} \in F$, $v$ and $\tilde{v}$ are disconnected. Otherwise, color the neighborhood set BLUE.
%\end{remark}

If $N_i$ is not colored RED then let $D_i$ and $F_i$ witness this failure. Note there are no edges in $D_i$.  Let $d_i = \max D_i \sqcup F_i$. Then, via the extension property, for all finite $D$ and $d$, where $d_i < \min D \leq \max D \leq d$, and there are no edges in $D_i \sqcup D$, then  there is a BLUE node $v$ to connected $D_i \sqcup D$ and not connected to $\{0, 1, \ldots, d \} - (D_i \sqcup D)$. %(In fact, every such $v$ is necessarily BLUE, but that is not used in the proof.)

If $N_i$ is colored RED then the same holds for RED nodes with $D_i = \{i\}$ and $d_i =i$. That is, for all finite $D$ and $d$, where $d_i < \min D \leq \max D \leq d$, and there are no edges in $D_i \sqcup D$, then there is a RED node $v$ to connected $D_i \sqcup D$ and not connected to $\{0, 1, \ldots, d \} - (D_i \sqcup D)$.

  Note that in both cases, for all $D$ and $d$, $v$ can be effectively (in the original coloring) found given $D_i$ and $d_i$.  Let $N_{D,d}$
denote the set of
all nodes connected to $D$ and not connected to $\{0,1, \ldots d\}-D$. 
So $N_{D_i,d_i} \subseteq N_i$.
%{\color{red} ND:  I don't see the ``So". This holds from the first sentence in this paragraph, since $D_i$ witnessing that $N_i$ is not colored RED implies that $i=\min(D_i)$.}
Note if $j> d_i$ then $d_i < \min D_j$.
%{\color{red}  (ND:  This is simply because $\min(D_j)=j$).} 

By the pigeonhole principle there must be a color $C$ with infinitely many $C$ colored neighborhood sets. We will use these neighborhood sets to construct a copy of $\mathbb{H}_3$ in color $C$.  Start by choosing a neighborhood set $N_{i_0}$ and a $C$ colored node $v_0 = \hat{v}_{j_0}=\hat{v}_{n_0,j_0}$ where $n_0 > d_{i_0} $ (such a node will exist in a large enough $C$ colored neighborhood set). So $v_0$ and no nodes in $D_{i_0}$ are connected. 

Assume that we have stagewise selected neighborhood sets $N_{i_l}$ and nodes ${v}_l = \hat{v}_{j_l}$ to ensure that, at stage $s$, the following hold:
\begin{enumerate}
    \item[(a)]
the nodes $v_0, \ldots, v_s$ and $\hat{v}_0, \ldots, \hat{v}_s$ are order-isomorphic as finite graphs;
\item[(b)]
all nodes and neighborhood sets have color $C$; and
\item[(c)]
 for $k< l \leq s$,
$d_{i_l} < j_l$,  
$\hat{v}_l$ is an exclusive neighbor of $\hat{v}_k$ iff ${v}_l$ is in $N_{D_{i_k},d_{i_k}}$, and, furthermore,
%{\color{red} (ND:  It's a bit imprecise here.  The $D_{i_k}, d_{i_k}$ are chosen for the BLUE case, but in the RED case, these sets are not given and we  have to proceed with some care.)} They are choosen read above.
if $\hat{v}_l$ is not an exclusive neighbor of $\hat{v}_k$ then ${v}_l$ will not be connected to any node in $D_{i_k}$. 
\end{enumerate}

At stage $s+1$, first pick an $N_{i_{s+1}}$ of color $C$ where $i_{s+1} > j_s$.
% %{\color{red}  (ND:  Note that if you  require $i_{s+1}>j_s$, then $v_{s+1}$ cannot be connected to $v_s$.   but it will not be the one where every node is connected with its successor in the enumeration. 
% The neighborhood set indices cannot grow too fast if you want $v_{s+1}$ to be able to have certain edges with the previous $v_l$'s. However, we could start with an enumeration of $H$ where each new neighborhood starts above the previous vertices.  Using the Extension Property  we can still cet a copy of $H$.  From there, we can get any enumeration as you mention earlier in the paper. 
% Henry and I thought through this.  I have a different construction written down on paper that makes a copy of any given enumeration, but it's more notationally cumbersome so I'd like to just make easy fixes to your write-up.  We got the idea for the general $n$, but $H_3$ needs to be cleaned up first.)}
Let $d = d_{i_{s+1}}$.  Let $k$ be the integer such that $\hat{v}_{s+1}= \hat{v}_{k,s+1}$.
So $\hat{v}_{s+1}$ is an exclusive neighbor of $\hat{v}_k$.  Let $D$ be the set of $j_l$ where $\hat{v}_{s+1}$ and $\hat{v}_l$ are connected. By our inductive hypothesis, we have $d_{i_k} < \min D \leq \max D < d$.  There are no edges in $D_{i_k}$.   Since we are working in $\mathbb{H}_3$, there are no edges in $D$ (otherwise the connected nodes and $\hat{v}_{s+1}$ form a triangle). By our inductive hypothesis about exclusive neighborhoods,  none of the nodes in $D$ have any connections with any of the nodes in $D_{i_k}$.  So there are no edges in $D_{i_k} \sqcup D$. Hence we can find a node $v_{s+1}=\hat{v}_{j_{s+1}}$ of color $C$, with $j_{s+1}> d_{i_{s+1}}$,  connected to $D_{i_k} \sqcup D$ and not connected to $\{0,1, \ldots d\}-(D_{i_k} \sqcup D)$ to continue our induction. 
\end{proof}

\begin{proof}[Needed modifications for $k\ge 3$ many colors]
	There are two possibilities. One is to modify the above proof by changing the extension coloring of neighborhood sets. Color $0$ as RED was done above.  Now color the neighborhood set $l+1<k$ if the neighborhood set does not have color $l$ and when looking for a node we can always find a node of color $l+1$. Otherwise color the neighborhood set color $k-1$.

The other is to just group the first $k-1$ colors and apply indivisibility for 2 colors. If the result copy is not homogeneous in the last color, take this copy, group the first $k-2$ colors and repeat. 
\end{proof}

\begin{corollary}\label{H3nodes}
	For every computable $k$-coloring of nodes of $\mathbb{H}_{3}$ there is a homogeneous copy of $\mathbb{H}_{3}$ computable in $\mathbf{0'}$. Moreover, for each computable $2$-coloring of the nodes of $\mathbb{H}_3$, there is either a computable homogeneous copy in the first color, or there is a homogeneous copy computable from $\mathbf{0}'$ in the second color. 
\end{corollary}

\begin{proof}
	Now we want to count quantifiers needed in the above proofs, assuming our original coloring is computable. We will start with $2$ colors. The construction of the homogeneous copy is computable from the infinitely many $D_i$ and $d_i$ which are used. It is enough to find these sets and numbers. 
	
	A neighborhood set being colored RED is a $\Pi^0_2$ property. BLUE is $\Sigma^0_2$.  So the extension coloring is computable in $\mathbf{0''}$. It takes another jump to decide if there are infinitely many RED neighborhood sets or infinitely many BLUE neighborhood sets. Even if we are told a neighborhood set is BLUE it takes $\mathbf{0'}$ in addition to find $D_i$ and $d_i$.  If colored RED then $D_i$ and $i$ are found easily.  

But we can do better. If there is a RED computable copy we are done, so now  assume not. We can assume that the greedy algorithm for constructing a RED copy (where all selected nodes must be RED, see Section~\ref{greedy}) never works. Then for all $d$, we can find a $D_i$ and $d_i$ such that all nodes in $N_{D_i,d_i}$ are BLUE and $d < \min D_i$.  Just start the greedy algorithm for a RED copy using nodes above $d$.  Then there is a stage in this algorithm where we cannot extend into some $N_{D_i,d_i}$ with a RED node.  We can find this stage using $\mathbf{0'}$.  

For arbitrary $k$, group the first $k-1$ colors.
If there is a computable copy with the first $k-1$ colors, repeat using this copy and one less color.  Otherwise there is a $\mathbf{0'}$ copy in the ungrouped color. 

We can make one slight improvement for $2$ colors. Forget the assumption there is not a RED computable copy. Start the greedy RED algorithm above $d$. Either it works and we get a RED computably homogeneous copy or we get the desired $D_i$ and $d_i$ using $\mathbf{0'}$ and repeat. 
\end{proof}

\section{Far anticliques}

\begin{definition}\label{farac}
	$A = \{ v_{n_i,j_i} | i < k \}$ is a \emph{far anticlique of size $k$ which is larger than $l$} iff $l < n_0 < j_0 < n_1 < j_1 <\ldots < n_{k-1} < j_{k-1}$. We call $\{n_0 < j_0 < n_1 < j_1 < \ldots< n_{k-1} < j_{k-1}\}$ the {\em set of indices} for this far anticlique of size $k$. We consider the empty set as the far anticlique of size $0$. 
\end{definition}

%\begin{corollary}\label{farAexistl}
%	For all $l$ and $k$, there exists a far anticlique of size $k$ which is larger than $l$. 
%\end{corollary}
%
%\begin{proof}
%	Part (1) of Lemma~\ref{local1} implies that for all finite graphs $H$ in $\tilde{\mathbb{H}}$ there is a one point extension $H \sqcup \{v_{n,j} \}$, where $n$ is greater than any of the indices in $H$ and there are no edges between any node in $G$ and  $v_{n,j}$. Moreover we can make this a \emph{large} one point extension since, for all $l$, we can insist that $n > l$. We can repeat this process $k$ times. If we start with $H = \emptyset$, a fixed $k$ and $l$, and apply this process $k$ times, we get a far anticlique of size $k$ which is larger than $l$.  
%\end{proof}

Let $\mathcal{F}^{\tilde{\mathbb{H}}}_{2k}\sse [\mathbb{N}]^{2k}$ be the  sets  of indices for all far anticliques of size $k$ inside our copy $\tilde{\mathbb{H}}$.  This set is computable in $\tilde{\mathbb{H}}$. We will show that $\mathcal{F}^{\tilde{\mathbb{H}}}_{2k}$ is largely $2$-extendable.
% This implies that far anticliques of size $k$ \emph{largely persist} inside $\tilde{\mathbb{H}}$. 
% The fact that far anticliques and other finite graphs  largely persist was proven 
% in \cite{Balko7} (see also \cite{DobrinenH_3ExactDegrees20}).
% The big difference here is that we {\em computably} find a largely $2$-extendable set $\mathcal{F}^{\tilde{\mathbb{H}}}_{2k}$ from $\tilde{\mathbb{H}}$. 
% In Section~\ref{triples} we will provide an example of a finite graph  where this does not seem possible.  

% {\color{red} Peter - do you want to give a precise definitin of "large persistence" or use the theorem statement as its definition?}

\begin{theorem}%[Large persistence of far anticliques]
\label{le4}
	For all $\tilde{\mathbb{H}}$ and $k$, $\mathcal{F}^{\tilde{\mathbb{H}}}_{2k}$ is a largely $2$-extendable family of tuples. 
\end{theorem}

\begin{proof}
	Let $\mathcal{F}= \mathcal{F}^{\tilde{\mathbb{H}}}_{2k}$. For each  $l_0$, we can apply Lemma~\ref{build} inductively $k$ times starting with $H=\emptyset$ to get far anticliques of size $k$ which are larger than $l_0$.  Hence 	$\mathcal{F}_1$ is infinite.  
	
	Apply Lemma~\ref{build} once to $H=\emptyset$ to get $n_0$ and infinitely many $j_0$ such that there is a node $v_{n_0,j_0}$.  Apply Lemma~\ref{build} inductively $k-1$ times starting with $H=\{ v_{n_0,j_0}\}$ to get far anticliques of size $k$.  In particular,  there is an $n_0$ with infinitely many $j_0$ such that $(n_0,j_0) \in \mathcal{F}_2$.  
	
	Let $\tilde{\mathcal{F}}_2 = \mathcal{F}_2$.  Inductively we will construct $\tilde{\mathcal{F}}_{2m}$, for $2 \leq m \leq k$, such that every tuple in $\tilde{\mathcal{F}}_{2m-2}$ ($\tilde{\mathcal{F}}_{2m-1}$) has large $1$-extensions in $\tilde{\mathcal{F}}_{2m-1}$ ($\tilde{\mathcal{F}}_{2m}$). Given a tuple $(n_0 < j_0 < \ldots < n_{m-2} < j_{m-2}) \in \tilde{\mathcal{F}}_{2m-2}$ apply Lemma~\ref{build} to the corresponding far anticliques of size $m-1$ given by $(n_0 < j_0 < \ldots < n_{m-2} < j_{m-2})$ and some $l$, to get infinitely many far anticliques of size $m$ corresponding to the tuple $(n_0 < j_0 < \ldots < n_{m-2} < j_{m-2}< n_{m-1} < j_{m-1})$ and $n_{m-1} > l$. Add all these resulting tuples to $\tilde{\mathcal{F}}_{2m}$.  Repeat for infinitely many $l$. 
	
	Let $\tilde{\mathcal{F}}= \tilde{\mathcal{F}}_{2k}$.  By Lemma~\ref{le1}, $\mathcal{F}$ is largely $2$-extendable.

\end{proof}

The following theorem is a special case of Theorem 9.9 in \cite{DobrinenJML23}.

\begin{theorem}[Far anticliques of size $k$ have the Ramsey Property in $\mathbb{H}_{n+1}$]\label{frd1}
	For each $k$ and for any finite coloring of  all far anticliques of size $k$ in our standard copy $\mathbb{H}_{n+1}$, there is a copy $\tilde{\mathbb{H}}$ of $\mathbb{H}_{n+1}$ inside the standard copy of $\mathbb{H}_{n+1}$ where all far anticliques of size $k$ have the same color. 
\end{theorem}

In this case we say that $\tilde{\mathbb{H}}$ is monochromatic.  When we refer to  $\tilde{\mathbb{H}}$ being monochromatic hopefully it will be clear which graph G we are coloring as well as which coloring we are referring to.

%Peter - I'm taking out the following remark because a) it is not necessary, and b) it does not clarify anything. After I make the appendix I can discuss the ideas of this remark there and add a reference to that discussion here.

%\begin{remark}
%{\color{red}  ND:  I will slightly revise this and talk about this set-up.}
%The reason for the simple formulation of Theorem \ref{frd1} for far anticliques (as opposed to the full statement of Theorem 9.9 in \cite{DobrinenJML23}) is that there is only one ``type" (called a ``diary" in \cite{Balko7}) of far anticlique in our standard copy, since each  node  in $\mathbb{H}_{n+1}$ forms  an $n$-clique with other nodes  in $\mathbb{H}_{n+1}$. (In  \cite{Balko7}, this property is called a {\em controlled coding triple}.)
%\end{remark}

\begin{theorem}\label{CodesJump}
	The statement that far anticliques of size $k$ have the Ramsey Property in $\mathbb{H}_{n+1}$ codes $\emptyset^{(2k-2)}$ and thus implies $ACA$ over $RCA_0$ when $k>2$.
\end{theorem}

\begin{proof}
	A coloring of $[\mathbb{N}]^{2k}$ induces a coloring on the far anticliques of size $k$ in any copy of $\mathbb{H}_{n+1}$. Color the far anticliques of size $k$ in the standard copy $\mathbb{H}_{n+1}$ via the Jockusch Coloring, $J^X_{2k}$ (Definition~\ref{JC}). Apply Theorem~\ref{frd1} to get a monochromatic $\tilde{\mathbb{H}}$ which is  a copy of $\mathbb{H}_{n+1}$. A coloring on the far anticliques of size $k$ in  $\tilde{\mathbb{H}}$ induces a coloring on the tuples in $\mathcal{F}^{\tilde{\mathbb{H}}}_{2k}$. Hence all the tuples in $\mathcal{F}^{\tilde{\mathbb{H}}}_{2k}$ have the same color w.r.t.\ $J_{2k}$.  By Theorem~\ref{le4}, $\mathcal{F}^{\tilde{\mathbb{H}}}_{2k}$ is a monochromatic largely $2$-extendable family of tuples. By Theorem~\ref{JCLE}, $\mathcal{F}^{\tilde{\mathbb{H}}}_{2k}$  computes  $X^{(2k-2)}$.
\end{proof}

\begin{definition}\label{equiv}
     Let $\forall C_0 \exists H_0 \mathcal{R}_0(C_0,H_0)$ and $\forall C_1 \exists H_1 \mathcal{R}_1(C_1,H_1)$ be two $\Pi^1_2$ sentences which are realized in $\mathcal{P}$.  We say $\forall C_0 \exists H_0 \mathcal{R}_0(C_0,H_0)$ is \emph{strongly Weihrauch reducible over $\mathcal{P}$} to $\forall C_1 \exists H_1 \mathcal{R}_1(C_1,H_1)$ iff there are two Turing functionals functionals $\Phi$ and $\Psi$ such that:
     \begin{itemize}
         \item  $\mathcal{P}$ realizes $$\forall C_0 \forall H_1 ( \mathcal{R}_1(\Phi(C_0),H_1)) \imply \mathcal{R}_0(C_0,\Psi(H_1)).$$  
         \item Both $\Phi(C_0)$ and $\Psi(H_1)$ are total functions on all reasonable oracle (sets) from $\mathcal{P}$.
     \end{itemize}
\end{definition}

 To solve $\mathcal{R}_0$ on $C_0$ is enough to find a solution for  $\mathcal{R}_1$ on $\Phi(C_0)$.  We have that $\Phi(C_0) \leq_T C_0$ and $\Psi(H_1) \leq_T H_1$.  For this reason, we will call $\Phi$ and $\Psi$ reductions.  Hence this is a strong way to show that if $\mathcal{P}$ realizes $\forall C_1 \exists H_1 \mathcal{R}_1(C_1,H_1)$ then $\mathcal{P}$ also realizes $\forall C_0 \exists H_0 \mathcal{R}_0(C_0,H_0)$.  If  $\forall C_0 \exists H_0 \mathcal{R}_0(C_0,H_0)$ is {strongly Weihrauch reducible over $\mathcal{P}$} to $\forall C_1 \exists H_1 \mathcal{R}_1(C_1,H_1)$ and $\forall C_0 \exists H_0 \mathcal{R}_0(C_0,H_0)$ codes $\emptyset^{(k)}$ then so does $\forall C_1 \exists H_1 \mathcal{R}_1(C_1,H_1)$. 

 \begin{theorem}\label{ramsey}
	$RT^l_k$ is strongly Weihrauch reducible to the statement that far anticliques of size $k$ have  the Ramsey property in $\mathbb{H}_{n+1}$.
\end{theorem}

\begin{proof}
	Let $C$ be a coloring of $[\mathbb{N}]^k$.  Every far anticlique $\mathcal{A}$ of size $k$ has an index set, $\{n_0 < j_0 < n_1 < j_1 < \ldots n_{k-1} < l_{k-1}\}$.  Color $\mathcal{A}$ with the color $C(n_0,n_1, \ldots n_{k-1})$. This reduction is total on all colorings $C$.  This is also a \emph{neighborhood set coloring} since it only depends on the neighborhoods.  Apply Theorem~\ref{frd1} to get a homogeneous $\tilde{\mathbb{H}}$ with color $\tilde{C}$.  Effectively in $\tilde{\mathbb{H}}$ and inductively we will construct a set $H$ of neighborhoods. 
    
    Let $\{n_0, j_0, n_1, j_1, \ldots ,  n_{k-1} , l_{k-1}\}$ be the index set for the first  far anticlique of size $k$ greater than $0$.  Such an anticlique exists by Theorem~\ref{le4}. Let $n_0 = n^0 \in H$. If we add $n^i$ to $H$ there is always a corresponding far anticlique $\mathcal{A}$ of size $k$ in $\tilde{\mathbb{H}}$ with index set $\{ n^i, j^i_0, n^i_1, j^i_1, \ldots ,  n^i_{k-1} , l^i_{k-1}\}$.  Given $n^i \in H$, find the first index set, $\{n^{i+1}_0, j^{i+1}_0, n^{i+1}_1, j^{i+1}_1, \ldots ,  n^{i+1}_{k-1} , l^{i+1}_{k-1}\}$, for a far anticlique of size $k$ greater than $l^i_{k-1}$. Again such an anticlique exists by Theorem~\ref{le4}. Let $n^{i+1} \in H$.  This reduction is total on all possible $\tilde{\mathbb{H}}$.
    
    Now take any subset of $H$ of size $k$, $\{ n^{i_0} < n^{i_1} < \ldots < n^{i_{k-1}}\}$. By our careful choices, $\{ n^{i_0} <  j^{i_0}< n^{i_1} < j^{i_1} < \ldots j^{i_{k_1}} < n^{i_{k-1}}  <  n^{i_{k-1}}\}$ is the index set of far anticlique of size $k$ with color $\tilde{C}$.  Hence  $\{ n^{i_0} < n^{i_1} < \ldots < n^{i_{k-1}}\}$ is colored $\tilde{C}$ by $C$. So $H$ is homogenous w.r.t.\ to $C$. 
	\end{proof}

\section{Close Pairs in $\mathbb{H}_3$}

\begin{definition}
	A pair of nodes, $v_{n_0,j_0}$ and $v_{n_1,j_1}$, are \emph{close} iff $n_1 < n_0 < j_0 < j_1$ (so all indexes are closely bound by $j_1$). 
	\end{definition}

 Close pairs may be connected or not.  For what follows it does not matter which, but we have to fix one of the two. From now on we work with (dis)connected close pairs. For far anticliques $n_i$ are increasing as $i$ increases. Here and from now on the $n_i$ are decreasing as $i$ increases.  
 
 Let $v_{l_0} \in \tilde{\mathbb{H}}$.   By Theorem~\ref{local1} (4), there is a node $v_n \in \tilde{\mathbb{H}}$ and infinitely many nodes $v_{n_i,j_i} \in \tilde{\mathbb{H}}$, where $l_0 < n_i \leq n$.  Let $\mathcal{F}^{\tilde{\mathbb{H}}}$ be the set of triples $(n_0,j_0,j_1)$  where $ v_{n_0,j_0}, v_{n_1,j_1} $ are a (dis)connected close pair in $\tilde{\mathbb{H}}$ and $n_0 > n$.  $\mathcal{F}^{\tilde{\mathbb{H}}}$  is computable in $\tilde{\mathbb{H}}$ and $n$ (which is just finitely much new information). So $\mathcal{F}^{\tilde{\mathbb{H}}}$  is computable in $\tilde{\mathbb{H}}$. 

\begin{theorem}%[Large persistence of close pairs]
\label{lpp}
	 $\mathcal{F}^{\tilde{\mathbb{H}}}$ is largely $2$- extendable.  
\end{theorem}

\begin{proof}
	Consider a node $v_{n_{0},j_{0}} \in \tilde{\mathbb{H}}$ where $n < n_{0}$.  For all such nodes and for all $l_0 > j_0$, by the extension property or Lemma~\ref{extdown}, there is another node $v_{n_{1},j_{1}} \in \tilde{\mathbb{H}}$ such that  $v_{n_{0},j_{0}}$ and $v_{n_{1},j_{1}}$ are (dis)connected, $v_{n_{1},j_{1}}$ and $v_n$ are connected (so $n_{1} \leq n < n_{0}$), and $l_0 < j_{1}$.  Hence the triple $(n_0,j_0,j_1)$ is in $\mathcal{F}^{\tilde{\mathbb{H}}}$.  
	
	By Theorem~\ref{local1} (2) there are infinitely many $n_{0} > n$ such that there is a $j_{0}$ where  $v_{n_{0},j_{0}} \in \tilde{\mathbb{H}}$. By Theorem~\ref{local1} (3) there is an $n_{0} > n$ such that there are infinitely many  $j_{0}$ where  $v_{n_{0},j_{0}} \in \tilde{\mathbb{H}}$. 	
\end{proof}

The following theorem for close connected pairs is a special case of Theorem 8.9 in \cite{DobrinenJML20}.  
For close disconnected pairs, the theorem is a special case of Theorem 4.5 in \cite{DobrinenH_3ExactDegrees20};
 it also follows from
 from the conjunction of  Theorem 8.9 in \cite{DobrinenJML20} and 
Theorem 3.4.12 in \cite{Balko7}. 

Note that this theorem only holds for $\mathbb{H}_{3}$.  We have more to say about close pairs in $\mathbb{H}_{n+1}$, when $n\ge 3$, in the next section.  In the notation of the next section, in $\mathbb{H}_{3}$, a close pair only has one diary and hence the next theorem holds. Far anticliques, close pairs in $\mathbb{H}_{3}$ and nodes  in $\mathbb{H}_{n+1}$ are the only cases where there is only one diary.

\begin{theorem}[close (dis)connected pairs have the Ramsey Property in $\mathbb{H}_{3}$]\label{frd3}
	For any finite coloring of close (dis)connected pairs in our standard copy $\mathbb{H}_{3}$, there is a copy $\tilde{\mathbb{H}}$ of $\mathbb{H}_{3}$ inside the standard copy of $\mathbb{H}_{3}$ where all close (dis)connected pairs have the same color. 
\end{theorem}

The proof of the following theorem is like the proof of Theorem~\ref{CodesJump}. 

\begin{theorem}\label{vcp0}
	The statement that ``close (dis)connected pairs have the Ramsey Property in $\mathbb{H}_3$" codes $\emptyset'$ and implies $ACA$ over $RCA_0$. 
\end{theorem}

\section{Finite graphs $G$}\label{finite}

Let $G$ be a finite $(n + 1)$-clique free graph. As we mentioned above, only if  $G$ is  a single node will it have the Ramsey Property in $\mathbb{H}_{n+1}$.  However, $G$ does have an exact finite big Ramsey degree in $\mathbb{H}_{n+1}$. That is, for each $G$ and $n$, there is a $k({G,n})$ such that for all finite colorings of copies of $G$ in $\mathbb{H}_{n+1}$ there is an $\tilde{\mathbb{H}}$ where at most $k({G,n})$ colors appear among the copies of $G$ in $\tilde{\mathbb{H}}$.  Moreover, for some coloring and all corresponding  $\tilde{\mathbb{H}}$, all $k({G,n})$ colors are needed. It turns out that the sentence 
``$G$ has 
finite
% exact
big Ramsey degree $k(G,n)$ in $\mathbb{H}_{n+1}$'' codes $\emptyset^{(|G|-1)}$.
But there are a number of things to prove before we get to this result.  The proof more or less follows the outline from the above two sections.  The tricky part is computably exacting from $\tilde{\mathbb{H}}$ the needed monochromatic largely $2$-extendable set. 

The exact big Ramsey degree of  $G$ in $\mathbb{H}_{n+1}$ is determined by the number  of essential patterns for copies of $G$ that persist in each subcopy of $\mathbb{H}_{n+1}$.
These patterns are called {\em diaries}, denoted by $\Delta$.
A diary for a finite graph $G$ is an 
enlargement of $G$ that notes essential properties about how $G$ sits inside $\mathbb{H}_{n+1}$.
Roughly, this includes the number of distinct neighborhoods in which $G$ lives and the first instances 
where subsets (with size between $1$ and $n$) of nodes in $G$ have
edges with all nodes in some  $k$-clique  in $\mathbb{H}_{n+1}$, where
  $1\le k< n$.

 Diaries for $\mathbb{H}_3$ have simple descriptions, appearing in 
\cite{DobrinenH_3ExactDegrees20}, \cite{Dobrinen_ICM}, \cite{Balko7}.
More generally,
diaries were 
defined abstractly  for a wide collection of 
free amalgamation classes with finitely many binary relations
in Definition 3.4.1 of \cite{Balko7}\footnote{This definition takes several pages.  For subgraphs of $\mathbb{H}_{n+1}$,
 Definition 9.4 in \cite{DobrinenJML23}, on which many ideas of  \cite{Balko7} were built, 
is much shorter.  That approach is modified in Definition \ref{def.diary} in the Appendix.}
 without reference to any particular copy of $\mathbb{H}_{n+1}$.
 In the Appendix, 
we give a simpler definition of diary for the  $(n+1)$-clique free graphs (see Definition \ref{def.diary}), constructible from the enumeration on $\mathbb{H}_{n+1}$.  In
 Theorem \ref{thm.diary}, we construct a particular diary, denoted $\mathbb{D}_{n+1}$,
 which  is   a subcopy $\widehat{\mathbb{H}}_{n+1}$ of $\mathbb{H}_{n+1}$ 
with some additional computable definable structure from $\mathbb{H}_{n+1}$ and its enumeration. The finite diaries $\Delta$ are substructures of the larger diary $\mathbb{D}_{n+1}$, so computably definable from $\mathbb{H}_{n+1}$.
The important point here is that  we may work in $\mathbb{D}_{n+1}$ or $\widehat{\mathbb{H}}_{n+1}$  interchangeably. 
Inside our standard enumeration $\mathbb{H}_{n+1}$ there are   copies of  $G$ that do not correspond to any diary. However, 
 every $G$ in $\widehat{\mathbb{H}}_{n+1}$  
 corresponds to exactly one diary. 
 In this section we will assume that we are coloring  graphs in $\widehat{\mathbb{H}}_{n+1}$ and that $\tilde{\mathbb{H}}$ is a copy of ${\mathbb{H}}_{n+1}$ inside $\widehat{\mathbb{H}}_{n+1}$.

We need the following definition:

\begin{definition}\label{defclose}
     A copy of a graph $G$, $\{v_{n_0,j_0},\dots, v_{n_{k-1},j_{k-1}}\}$, is {\em close} if $n_0 < j_0$, $\max \{ n_i \} = n_0$, and $j_0 < j _1 < \ldots <j_{k-1}$.   
\end{definition}

This definition extends the definition of close pairs. There is no implied order on the $n_i$'s other than  $n_0$ is the largest. Fix a finite  $(n+1)$-clique free $G$.  It follows from  Definition 3.4.1   and Theorem 3.4.12 in \cite{Balko7} that in every copy of $\mathbb{H}_{n+1}$ contains    copies of $G$ that are close. However, a much simpler proof of this fact appears in the proof of  Theorem \ref{thm.closele} below.  We will really only need that a copy of a  graph is close in the proof of Theorem~\ref{thm.closele}.  %, rather than referring to diaries. 
Determining whether a copy of $G$ in $\tilde{\mathbb{H}}$ is close is uniformly computable, just by looking at the indexing. 

%From each diary for $G$ we need to extract a tuple.
If  $\Delta$ is a diary for $G$,
then the maximal nodes in $\Delta$ form a copy of $G$.
%Inside each $\Delta$ there is a copy of $G$, $\{v_{j_0},\dots,v_{j_{|G|-1}}\}$.
Depending on the diary, this copy may or may not be close.  I.e., for some diaries $\Delta$,
all graphs with that diary are close;
for other diaries, all graphs with that diary are not close.  It is easy to see which is which, again by looking the indexing. 
We can assume the indices of the nodes in a particular copy of  $G$ are  $\{j_0<j_1<\dots<j_{|G|-1}\}$. 
Let $m$ be the greatest among the $n_i$.
Here we are interested in the $|G|+1$-tuple $\{m,j_0,j_1,\dots,j_{|G|-1}\}$.
When our copy of $G$ is close, then  $j_0> m$ and 
 $m=n_0$, so
this tuple is $\{n_0,j_0,j_1,\dots,j_{|G|-1}\}$.  

The next theorem 
is proved at the end of the Appendix, Section~\ref{appendix}.

\begin{theorem}\label{dairies}
There is a computable function $k(G,n)$ such that there are $k(G,n)$ diaries corresponding to copies of $G$ in $\mathbb{H}_{n+1}$. 
There is a  copy $\widehat{\mathbb{H}}_{n+1}$, computable in $\mathbb{H}_{n+1}$,  so that 
each copy of $G$ in $\widehat{\mathbb{H}}_{n+1}$, and hence in any subcopy $\tilde{\mathbb{H}}$, corresponds to exactly one diary. Moreover, there is a computable function $D(G,\tilde{\mathbb{H}})$ which outputs this diary  (using $\tilde{\mathbb{H}}$ as an oracle).  In particular, diaries for close copies of $G$ are different from the diaries for copies of $G$ which are not close. 
\end{theorem}

In terms of the Ramsey Property for diaries needed here, Theorem 9.9 in \cite{DobrinenJML23} suffices for our uses in this section, but in order not to introduce more  definitions we state the following special case of Theorem 5.1.5  
in \cite{Balko7}.

\begin{theorem}[Ramsey Property for diaries in $\mathbb{H}_{n+1}$]\label{thm.FG}
	For all finite $G$ and any diary $\Delta$ of $G$, for any finite coloring of all copies of $\Delta$ in  $\widehat{\mathbb{H}}_{n+1}$, there is a copy $\tilde{\mathbb{H}}$ of $\widehat{\mathbb{H}}_{n+1}$  in which  all copies of $\Delta$ have the same color. 
\end{theorem}

From now on we can equate $\widehat{\mathbb{H}}_{n+1}$ and $\mathbb{H}_{n+1}$.  Applying the above theorem $k(G,n)$ times yields the following corollary:

\begin{corollary}
     Let $\Delta_i$, for $i<k(G,n)$, list all the diaries for $G$. Then $$\mathcal{P}=2^\mathbb{N} \models \forall C \exists \tilde{\mathbb{H}} \bigwedge_{i<k} \mathcal{R}_{\Delta_i,\mathbb{H}_{n+1}}(C,\tilde{\mathbb{H}}).$$
\end{corollary}

%\pc{Can we let $\widehat{H}$ and $\mathbb{D}$ be the same structure? If so I want to rewrite this section a bit.} 
%{\color{blue}Technically a diary has a lot more structure than the graph it encodes.  But every diary has the property that the maximal nodes encode a graph.The terminal nodes of $\mathbb{D}$ form a copy of $\mathbb{H}$ that we can call $\widehat{\mathbb{H}}$.}
%By refining  a construction in Lemma 10.1 of \cite{DobrinenJML23} and adding ``controlled coding triples'' (Definition 2.2.2 in \cite{Balko7}), we obtain the following theorem.

 The following theorem is also proved in  the Appendix, Section~\ref{appendix}; it is (6) of Theorem \ref{thm.diary}.

\begin{theorem}\label{lem.D}
Fix $n\ge 2$. There is a computable diary $\mathbb{D}_{n+1}$ for $\mathbb{H}_{n+1}$ with the following property: For each finite $(n+1)$-clique free graph $G$, for each diary $\Delta$ for $G$, the set of tuples $\{m,j_0,j_1,\dots,j_{|G|-1}\}$ corresponding to copies of $\Delta$ in $\mathbb{D}_{n+1}$ is largely $2$-extendable.  In particular, for each diary $\Delta$ for a close copy of a graph $G$, the set of all tuples $\{n_0,j_0,j_1,\dots,j_{|G|-1}\}$ in $\mathbb{D}_{n+1}$ is largely $2$-extendable. 
\end{theorem}

We point out that a subdairy of $\mathbb{D}_{n+1}$ determines an unique subcopy of $\mathbb{H}_{n+1}$.  We may equate the corresponding subdairy and subcopy of ${\mathbb{H}}_{n+1}$ since each subcopy of ${\mathbb{H}}_{n+1}$ has a unique computable expansion to a diary. The next theorem is a special case of Theorem 3.4.12 in \cite{Balko7}:

\begin{theorem}\label{thm.B7}
Any two diaries for 
 $\mathbb{H}_{n+1}$ are bi-embeddable as diaries.
\end{theorem}

 Theorem \ref{thm.B7} implies that for each  copy  $\tilde{\mathbb{H}}$  of $\mathbb{H}_{n+1}$,
$\mathbb{D}_{n+1}$ embeds into $\tilde{\mathbb{H}}$.  Theorems~\ref{lem.D} and \ref{thm.B7}  imply the following:

\begin{corollary}\label{RPd}
    For each finite diary $\Delta$, $\tilde{\mathbb{H}}$ contains a set of tuples corresponding (as in Theorem~\ref{lem.D}) to copies of  $\Delta$ which is largely $2$-extendable.  
\end{corollary}

However the copy of $\mathbb{D}_{n+1}$ in   $\tilde{\mathbb{H}}$ and this largely $2$-extendable set are not necessarily computable from  $\tilde{\mathbb{H}}$.  %\footnote{In the future, we hope that Theorem \ref{thm.B7} will be proved by simpler means for special cases such as close graphs, similarly as we have  done for far anticliques in Theorem \ref{le4}.}  
%{\color{blue} (The copy of $\mathbb{D}$ inside $\tilde{\mathbb{H}}$ is computable from $\emptyset''$ relative to $\tilde{\mathbb{H}}$. But that does not help us here.)}
Hence, we will need another argument to show, in Theorem~\ref{thm.closele}, that there is family corresponding to close copies of $G$ which is both largely $2$-extendable and computable in $\tilde{\mathbb{H}}$.

% Let $G=\{g_0,\dots,g_{k-1}\}$ be a finite  $(n+1)$-clique free graph , where $k\ge 2$. Let $\tilde{\mathcal{F}}\sse[\mathbb{N}]^{k+1}$
% be the set of all tuples $\{n_0<j_0<j_1<\dots<j_{k-1}\}$ where $\{j_0 < j_1 <  \dots < j_{k-1}\}$
% is a close copy of $G$ in the original computable copy $\mathbb{H}$.  $\tilde{\mathcal{F}}$ is computable, since the computable copy $\mathbb{H}$ can tell exactly when a tuple $\{j_0<j_1<\dots<j_{k-1}\}$ is a close  copy of $G$.  Note that the tuple $\{n_0<j_0<j_1<\dots<j_{k-1}\}$ can be computably and uniformly found from any dairy of close copies of $G$. 
% If we are dealing with a dairy which not a dairy of a close copy of $G$ we will let find a tuple $\{n_0<j_0<j_1<\dots<j_{k-1}\}$ by finding $\{v_{j_0},\dots,v_{j_{|G|-1}}\}$ inside the dairy which is isomorphic to $G$. We will use these tuples to color these dairies. 

\begin{theorem}\label{thm.closele}
	For each copy  $\tilde{\mathbb{H}}$  of $\mathbb{H}_{n+1}$ and each finite $(n+1)$-clique free graph $G$ with at least two vertices, there is  a family  $\mathcal{F}$ of tuples corresponding (as above) to close copies of $G$ which is  largely $2$-extendable  and computable from $\tilde{\mathbb{H}}$.
\end{theorem}

\begin{proof}

We will inductively build $\mathcal{F}$.  We will start with $\mathcal{F}_2$ and then, given $\mathcal{F}_{l}$, construct $\mathcal{F}_{l+1}$. Each $\mathcal{F}_l$ will be largely $2$-extendable, computable from $\tilde{\mathbb{H}}$, and the tuples in $\mathcal{F}_l$ will be the indexes for a close copy of the graph, $\{ g_0, g_1, \ldots , g_{l-2}\}$. We will let $\mathcal{F}=\mathcal{F}_{k+1}$. 

By Theorem \ref{le4}, there are far anticliques of all sizes in $\tilde{\mathbb{H}}$.
Fix a far anticlique of size $k-1$, $\{ v_{\tilde{n}_i,x_i}  :  0\leq i \leq k-2\}$, in $\tilde{\mathbb{H}}$.

Let $\mathcal{F}_2$ be the set of indices of nodes $v_{n_0,j_0}$ in  $\tilde{\mathbb{H}}$ where $n_0 > x_{k-2}$. 
By Theorem \ref{local1} (1) and (3), $\mathcal{F}_2$ is largely $2$-extendable.  Given $x_{k-1}$ and $\tilde{\mathbb{H}}$, $\mathcal{F}_2$ is computable. Since the pairs in  $\mathcal{F}_2$ correspond to nodes, they are automatically close. Hence the inductive hypotheses are met.

We will assume $\mathcal{F}_{l}$ is given meeting the inductive hypotheses, and, additionally, if $(n_0,j_0,\ldots ,j_r,\dots, j_{l-2})\in{\mathcal{F}_l}$ then, for all $r$, if $1\leq r < l-2$, then $v_r$ is connected to exactly one node in the given anticlique, $v_{\tilde{n}_{k-1-r},x_{k-1-r}}$.  Let $(n_0,j_0,\dots, j_{l-2})\in{\mathcal{F}_l}$.  Then, for each $l_1 > j_{l-2}$, by the Extension Property, there are infinitely many $j_{l-1} > l_1$ in $\tilde{\mathbb{H}}$ such that $\{ v_{j_0}, v_{j_1}, \ldots v_{j_{l-2}}, v_{j_{l-1}}\}$ is order isomorphic to $\{ g_0, g_1, \ldots, g_{l-1}\}$ and, among the given anticlique, $v_{j_{l-1}}$ is only connected to $v_{\tilde{n}_{k-l},x_{k-l}}$. For such a tuple $n_0 > n_{j_{l-1}}$ and $j_{l-2} <j_{l-1}$.  Hence the tuple $\{ v_{j_0}, v_{j_1}, \ldots, v_{j_{l-2}}, v_{j_{l-1}}\}$ is a close copy of $\{ g_0, g_1, \ldots, g_{l-1}\}$. Let $\mathcal{F}_{l+1}$ be the set of all such tuples for all $l_1$.  $\mathcal{F}_{l+1}$ is computable given $\mathcal{F}_{l}$ and, by Lemma~\ref{le1}, $\mathcal{F}_{l+1}$ is largely $2$-extendable.  

So $\mathcal{F}$ is the set of tuples $(n_0,j_0,\ldots ,j_r,\dots, j_{k-2})$ where $n_0 > x_{k-2}$, the graph $\{ v_{j_0}, v_{j_1}, \ldots v_{j_{k-1}}\}$ is a close  copy of $\{ g_0, g_1, \ldots, g_{k-1}\}$, and if $1\leq r < l-2$, then $v_r$ is connected to exactly one node in the given anticlique,  namely $v_{\tilde{n}_{k-1-r},x_{k-1-r}}$. 
\end{proof}

\begin{theorem}\label{thm.onecolornearG}
    Let $\Delta_i$, for $i<k=k(G,n)$, be all the diaries for $G$. Then the $\Pi^1_2$ sentence  ``$\forall C \exists \tilde{\mathbb{H}} \bigwedge_{i<k} \mathcal{R}_{\Delta_i,\mathbb{H}_{n+1}}(C,\tilde{\mathbb{H}})$" codes $\emptyset^{|G|-1}$. 
\end{theorem}

% \begin{theorem}
% For each $n\ge 2$, for each  finite  $(n+1)$-clique free graph $G$, ($|G| > 1$), 
% there is a computable $2$-coloring on the set of all close copies of $G$ in $\mathbb{H}_{n+1}$ such that there is a $\tilde{\mathbb{H}}$ where all close copies of $G$ have the same color and moreover any such $\tilde{\mathbb{H}}$ computes $\emptyset^{|G|-1}$.  
% \end{theorem}

\begin{proof}
Color all diaries of $G$ with the Jockusch coloring $J_{|G|+1}$. So there is a copy $\tilde{\mathbb{H}}$ in which each diary of $G$ has one color.  By Corollary~\ref{RPd}, for each diary $\Delta$ for $G$ there is a largely $2$-extendable subset encoding copies of $\Delta$. By Theorem~\ref{JCLE} (1)  for each diary of a close copy of $G$, the  copies of that diary in $\tilde{\mathbb{H}}$  have $J_{|G|+1}$-color $1$.  This implies in $\tilde{\mathbb{H}}$ all close copies of $G$ have color 1 w.r.t.\ to  $J_{|G|+1}$. Use Theorem~\ref{thm.closele} to get a family $\mathcal{F}$ of  close copies of $G$ in $\tilde{\mathbb{H}}$ that is  largely $2$-extendable  and computable from $\tilde{\mathbb{H}}$.  This set is monochromic w.r.t.\  $J_{|G|+1}$.  Now apply Theorem~\ref{JCLE} (3) for our conclusion. 
\end{proof}

\begin{lemma}
    Let $\Delta_i$, for $i<k=k(G,n)$, be all the diaries for $G$.  The sentence 
    ``$\forall C \exists \tilde{\mathbb{H}} \bigwedge_{i<k} \mathcal{R}_{\Delta_i,\mathbb{H}_{n+1}}(C,\tilde{\mathbb{H}})$" is strongly Weihrauch reducible to ``$G$ has finite big Ramsey degree in $\mathbb{H}_{n+1}$''. 
\end{lemma}

\begin{proof}
    Assume we are given a finite coloring $C$ of all the copies of these diaries in $\widehat{\mathbb{H}}_{n+1}$. By Theorem~\ref{dairies}, each copy of $G$ computably extends to a unique $\Delta_i$. Color that copy of $G$ by the pair $(i,{C}(\Delta_i))$.  Use the sentence ``$G$ has finite big Ramsey degree in $\mathbb{H}_{n+1}$'' to get a $\tilde{\mathbb{H}}$ where are at most $k$ colors are used; one for each diary.  
    Each diary must appear in $\tilde{\mathbb{H}}$, see Corollary~\ref{RPd}.  So each $i<k$ appears in the first component of our coloring.  Hence, for all $i<k$, the second component must be constant on all copies of the diary $\Delta_i$.  Note for two different diaries the value of the second component need not be the same.  
\end{proof}

So, by Corollary~\ref{RPd}, $\mathcal{P}=2^{\mathbb{N}}$ realizes ``$G$ has finite big Ramsey degree $k(G,n)$ in $\mathbb{H}_{n+1}$''.  Our ultimate result now follows from Definition~\ref{equiv} and Theorem~\ref{thm.onecolornearG}. 

\begin{theorem}
    ``$G$ has finite big Ramsey degree $k(G,n)$ in $\mathbb{H}_{n+1}$'' codes $\emptyset^{|G|-1}$ and implies $ACA_0$. 
\end{theorem}

But we have a little more to say about this situation: 

\begin{lemma}
    Let $\Delta_i$, for $i<k=k(G,n)$, be all the diaries for $G$.  The sentence ``$G$ has finite big Ramsey degree $k(G,n)$ in $\mathbb{H}_{n+1}$''  is strongly Weihrauch reducible to the sentence ``$\forall C \exists \tilde{\mathbb{H}} \bigwedge_{i<k} \mathcal{R}_{\Delta_i,\mathbb{H}_{n+1}}(C,\tilde{\mathbb{H}})$".
\end{lemma}

\begin{proof}
    A coloring of copies of $G$ clearly induces a coloring of all diaries of $G$.  This is well-defined by Theorem~\ref{dairies}.  Use the sentence ``$\forall C \exists \tilde{\mathbb{H}} \bigwedge_{i<k} \mathcal{R}_{\Delta_i,\mathbb{H}_{n+1}}(C,\tilde{\mathbb{H}})$" to get a copy of $\tilde{\mathbb{H}}$  in which each diary for a copy of $G$ has one color.  In $\tilde{\mathbb{H}}$, again by Theorem~\ref{dairies}, each copy of $G$ gets one color.  At most $k$ colors are used in $\tilde{\mathbb{H}}$.  
\end{proof}

\begin{definition}
     Let $\varphi_0$ and $\varphi_1$ be sentences which are realized in $\mathcal{P}$.  We say $\varphi_0$ and $\varphi_1$ are \emph{strongly Weihrauch equivalent over $\mathcal{P}$}  iff  $\varphi_0$ is {strongly Weihrauch reducible over $\mathcal{P}$} to $\varphi_1$ and $\varphi_1$ is {strongly Weihrauch reducible over $\mathcal{P}$} to $\varphi_0$. 
\end{definition}

\begin{lemma}
    The two sentences ``$\forall C \exists \tilde{\mathbb{H}} \bigwedge_{i<k} \mathcal{R}_{\Delta_i,\mathbb{H}_{n+1}}(C,\tilde{\mathbb{H}})$" and ``$G$ has finite big Ramsey degree $k(G,n)$ in $\mathbb{H}_{n+1}$'' are strongly Weihrauch equivalent. 
\end{lemma}

%%%%  I had to rewrite this.  For the 1st we need to be more careful. The 2nd part is not correct. You need a computable set of tuples of G and Delta to say that and we do not have that.  And you merged the two proofs of 8.2 and 8.3. You also used "near" rather than "close" several times. 

% Lemma \ref{JCLE} and
%  Theorems
% \ref{thm.FG},
% \ref{thm.B7}, and 
% \ref{thm.closele}
% imply the following:

% \begin{theorem}\label{thm.onecolornearG}
% For each $n\ge 2$, for each  finite  $(n+1)$-clique free graph $G$, ($|G| > 1$), 
% there is a computable $2$-coloring 
% on the set of all close copies of $G$ in $\mathbb{H}_{n+1}$ so that any homogeneous $\tilde{\mathbb{H}}$ codes 
% $\emptyset^{|G|-1}$.  
% Furthermore, the statement 
% ``For all finite $G$, the each diary of $G$ has the Ramsey property in $\mathbb{H}_{n+1}$ and any two diaries for $\mathbb{H}_{n+1}$ are bi-embeddable"
% implies $ACA$ over $RCA_0$.
% %	For all finite $G$, ($|G| > 1$) the statement that $\mathcal{H}^{\tilde{G}}$ has the Ramsey Property  in $\mathbb{H}_{n+1}$ codes $\emptyset'$.  Hence this statement implies $ACA$ over $RCA_0$.
% \end{theorem}

%%%  I had to adjust this.

\subsection{Some remaining technical  questions}

All these questions are about possible improvements to Theorem~\ref{thm.closele} as there is plenty of room to improve Theorem~\ref{lem.D}. Given the nodes and neighborhoods involved in a diary we can extract a large tuple.  For example, just involving all the neighborhood sets would give a $(2|G|-2)$-tuple, $(n_{|G|-1}, n_{|G|-2}, \ldots n_0,j_0, \ldots j_{|G|-1})$. But even larger tuples are possible, depending on $G$ and $n$. 
(See (6) of Theorem \ref{thm.diary}.)
Now  Theorem~\ref{lem.D} holds  for this set of tuples and more.  
%{\color{blue} (ND:  All of this follows from my construction.  I will state it explicitly and prove it in the Appendix. Meaning the thing you want to improve here already follows from the work I did.  So I will add more exposition to  make it clear.)}
It is open whether we can do the same with Theorem~\ref{thm.closele}.  But if we can make this improvement then the set computed in Theorem~\ref{thm.onecolornearG} would be  $\emptyset^{(2|G|-2)}$ (or more) rather than $\emptyset^{(|G|-1)}$.

Given an $(m+1)$-clique free $G$ with at least two vertices, 
the number and the complexity of the diaries for $G$  grow as $n\ge m$ grows. It thus seems likely that there is some computable increasing unbounded sequence $(k_n)_{n\ge m}$ so that for each $n\ge m$ the statement, ``$G$  has finite big Ramsey degree in $\mathbb{H}_{n+1}$"
codes $\emptyset^{(k_n)}$.
At this time, this remains an open question. This involves extracting tuples of size $k_n+2$ from the diaries and using these tuples in Theorems~\ref{lem.D} and \ref{thm.closele} as suggested above.

Fix a diary $\Delta$ for $G$. If we could restricted the copies of $G$ used in Theorem~\ref{thm.closele} to those with fixed single diary $\Delta$ then we could positively answer the following question:  

\begin{question}
For each finite $G$ and for each diary $\Delta$ for $G$, does the sentence ``The Ramsey property holds for $\Delta$  
in $\mathbb{H}_{n+1}$" code $\emptyset^{(|G|-1)}$? 
%	Is there some way to prove for all finite $G$, the statement that $\mathcal{H}^{\tilde{G}}$ has the Ramsey Property  in $\mathbb{H}_{n+1}$ codes $\emptyset^{(k)}$, where $k = |{\tilde{G}}|+1$?   
\end{question}

\section{Coloring Nodes}\label{nodes}

\subsection{Neighborhood set colorings of nodes} 

This section presents mostly work due to Gill \cite{Gill23} but in a slightly different format and with some different questions. Critical to this section is the following theorem:

\begin{theorem}[Folkman \cite{MR268080}]
    For all $n>1$, for all $l$, there is a finite connected graph $G^{n,l}$ that does not contain a copy of $K_{n+1}$ and, for each $l$-coloring of $G^{n,l}$, there is a color $C$ where within the nodes of $G^{n,l}$ with color $C$ there is a copy of $K_n$.  
\end{theorem}

  We will choose an enumeration $\{g_0,g_1, \ldots \}$ of $G^{n,l}$ such that, for all $s$, the graph $\{g_0,g_1, \ldots, g_s \}$ is connected.  Since in $\mathbb{H}_{n+1}$ no neighborhood contains a copy of $K_n$, inside $\tilde{\mathbb{H}}$, $G^{n,l}$ cannot be $l$ colored by the identity neighborhood set coloring, i.e., the coloring where each neighborhood $N_i$ is colored $i$.

  \begin{corollary}
       If we use the greedy algorithm to construct a copy of $G^{n,l}$ inside $\tilde{\mathbb{H}}$,  at some stage we must select a node $v_{n,j}$ where $n>l$.
  \end{corollary}
    
 The neighborhood set coloring  is strictly determined by the neighborhoods.   We call a node coloring is a \emph{stable neighborhood set coloring} iff, for each neighborhood set  $N_i$, almost all nodes in $N_i$ get the same color.

\begin{theorem}[Gill \cite{Gill23}]
 There is a stable neighborhood set $2$-coloring of  $\mathbb{H}_{n+1}$ with no computable (c.e.) homogeneous copy of $\mathbb{H}_{n+1}$. 
 \end{theorem}
 
 \begin{proof}
   This is a stagewise  argument. At stage $s$ we will color $v_s = v_{n,s}$ by the current default coloring of the neighborhood $N_n$. Start all neighborhood sets with the default color BLUE. The $e$th requirement, $\mathcal{R}_e$, is to ensure that $W_e$ is not a homogeneous copy of $\mathbb{H}_{n+1}$. If, at any stage $s$, $W_{e,s}$ is not homogeneous, $\mathcal{R}_e$ is met and no longer acts. Assume all $x \in W_{e,s}$ have color BLUE (if RED the coloring needed is dual). 
   
   With the exception of its first node, $\mathcal{R}_e$ must work with the neighborhoods above $l_{e,s}$. For each $e$, let $l_{e,0}=0$. $\mathcal{R}_e$ will be tasked with the job of using the greedy algorithm, see Section~\ref{greedy}, to construct a copy of $G^{n ,l(e,s)}$ inside $W_e$.  At stage $s$, $\mathcal{R}_e$ may only select nodes from $W_{e,s}$ and nodes which are larger than the threshold function $t_{e,s}$. Let $t_{e,0}=0$.    At a stage $s$ where such a selection for $\mathcal{R}_e$ is possible and $e$ is the smallest $\tilde{e}$ where such a selection is possible, we select a new node $v_{j}$ for $\mathcal{R}_e$'s copy of $G^{n,l(e,s)}$, change the default color of $N_i$ to RED, for all $i$ such that $l(e,s)<i\leq j$, let $t_{e,s+1}=s+1$, and for all $e' > e$, let $l(e',s+1)=t_{e',s}= s+1$ and $\mathcal{R}_{e'}$ needs to restart its construction of $G^{n,l(e',s)}$ with all new nodes. Moreover if this is the first node that $\mathcal{R}_e$ has chosen, we will also let $l(e,s+1)=s+1$ (this single node can be used in the greedy algorithm for the construction of any finite $G$). We say that $\mathcal{R}_e$ \emph{acts} and that $\mathcal{R}_{e'}$ was \emph{injured}.  
   
    Assume $\mathcal{R}_e$ is never injured after some stage $t$.  Now only consider stages $s > t$.  Then $l(e,s)$ can only increase once more than the first node of $G^{n,l(e,s)}$ is chosen. Hence $\lim_s l(e,s)= l(e)$ exists. Since $G^{n,l(e)}$ is finite,  $\mathcal{R}_e$ only acts finitely often and hence $\mathcal{R}_{e+1}$ is only injured finitely often.  By the above Corollary, at some stage $s$, $\mathcal{R}_{e}$ must select a node $v_{n,j}$ where $n>l(e)$. This cannot be the first node selected. By our enumeration, there must be an earlier stage $\tilde{s}$ where a node $v_{\tilde{j}}$ which is connected to $v_{n,j}$ has been chosen. After stage $\tilde{s}$, all the nodes $v_{n',j'}$ where $l(e) < n' \leq \tilde{j}$ and $t_{e,\tilde{s}}< j'$, are colored RED. Hence $v_{n,j}$ must be colored RED.  Hence $\mathcal{R}_e$ is met. 
  \end{proof}

\subsection{Questions about colorings of nodes}

We can turn a (stable) coloring of pairs, $C(n,s)$, into a (stable neighborhood) coloring by coloring the node $v_{n,s}$ by $C(n,s)$.  Assume that we have a homogeneous copy $\tilde{\mathbb{H}}$ for a stable coloring of pairs, $C(n,s)$. If we could improve Theorem~\ref{local1} (3) to effectively (in $\tilde{\mathbb{H}}$ ) find ${n}$, then we could compute a homogeneous set for this stable coloring of pairs using $\tilde{\mathbb{H}}$. Just repeatedly apply Theorem~\ref{local1} (3) to get an infinite set of these ${n}$ and then thin to get a homogeneous set. But, by Lemma~\ref{notcomputable}, this improvement in Theorem~\ref{local1} (3) is not possible. So how to use $\tilde{\mathbb{H}}$ to find homogeneous set for this stable coloring is still open. 

\begin{question}
	Is $(S)RT^2_2$ (strongly) Weihrauch reducible to the statement that "nodes have the Ramsey Property in $\mathbb{H}_{3}$"?
\end{question}

 Corollary~\ref{H3nodes} tells us that if we have a computable coloring of nodes in $\mathbb{H}_3$, then we get a $0'$-computable homogeneous set.  It is not known if we can do better.

\begin{question}[Gill]
	Does the statement that ``nodes have the Ramsey Property in $\mathbb{H}_{3}$" satisfy cone avoidance?  Where in the reverse mathematics zoo does this statement live?  
\end{question}

By the extension property, every node in $\mathbb{H}_{n+1}$ can be largely extended to be in a $k$-clique in $\mathbb{H}_{n+1}$, where $k\leq n$.  We will use this to identify some nodes with $k$-tuples. Assume $v_i$ is the largest node in some $k$-clique, $\{ v_{i_0} < v_{i_1} < \ldots < v_{i_{k-1}} = v_i \}$ and, for any other $k$-clique, $\{ v_{\tilde{\imath}_0} < v_{\tilde{\imath}_1} < \ldots < v_{\tilde{\imath}_{k-2}} < v_{i_{k-1}} = v_i \}$, the $k$-tuple $\{i_0 < i_1 < \ldots <i_{k-2} < i \}$ is lexicographically below the $k$-tuple $\{\tilde{\imath}_0 < \tilde{\imath}_1 < \ldots <\tilde{\imath}_{k-2} < i \}$. In this case, we say that $i$ codes the $k$-tuple $\{i_0 < i_1 < \ldots <i_{k-2} < i \}$. One can show that the set of these codes is largely $2$-extendable in the standard copy of $\mathbb{H}_{n+1}$. If we could do the same in an arbitrary subcopy $\tilde{\mathbb{H}}$, that would prove a result analogous to Theorem~\ref{CodesJump}.   But unfortunately we are just left with a question: 

\begin{question}
	Does the statement that ``nodes have the Ramsey Property in $\mathbb{H}_{n+1}$" code $\emptyset^{(n-1)}$ and imply $ACA$ over $RCA_0$?
\end{question}

\section{Appendix:  Diaries for the Henson graphs}\label{appendix}

This appendix provides a constructive  definition of diary for the Henson graphs and includes a construction of a  diary satisfying Theorems  \ref{dairies} and \ref{lem.D}.
The following sort of enumeration was used in  \cite{DobrinenJML20} and \cite{DobrinenJML23} and has been used throughout this paper.

\begin{definition}[Standard Enumeration of $\mathbb{H}_{n+1}$]\label{defn.standardenum}
Fix $n\ge 2$. 
A {\em standard enumeration} of
$\mathbb{H}_{n+1}$  is a  computable enumeration with the following additional property:
For each pair  $i<j<n$, $\hat{v}_j$ is connected to $\hat{v}_i$.
For each $j\ge n$, 
$\hat{v}_j$ forms an $n$-clique with its $n-1$ immediate predecessors; that is, 
all pairs of nodes among 
$\{\hat{v}_{j-n+1}, \dots,\hat{v}_j\}$ are connected. 
\end{definition}

This enumeration ensures  that  whenever a graph $G$ consists of nodes $\hat{v}_j$ with  all $j\ge n$,  then  each  node in $G$ forms an $n$-clique with nodes in  $\mathbb{H}_{n+1}$.
In particular, $\hat{v}_{j+1}$ is connected to $\hat{v}_j$ for each $j\ge 0$.
This serves to reduce upper bounds on big Ramsey degrees from the outset and is a component of the notion of diary, below.

The order on the nodes in  $\mathbb{H}_{n+1}$ induces a tree of binary sequences as follows.
Given $j<\om$, we associate  $\hat{v}_j$  with the
binary  sequence of length $j$ such that for $i<j$,  the $i$-th entry of $\hat{v}_j$ is $1$ iff $\hat{v}_j$ is connected to $\hat{v}_i$.
Abusing notation, we  denote this sequence again by $\hat{v}_j$.
Then
for each $i<j$, $\hat{v}_j(i)=1$ iff  $\hat{v}_j$ is connected to $\hat{v}_i$
and 
$\hat{v}_j(i)=0$ iff  
$\hat{v}_j$ is disconnected from $\hat{v}_i$.
Let $T\sse 2^{<\om}$ denote the collection of all initial segments of the binary sequences $\hat{v}_j$, $j<\om$.
Then  $T$ is a tree, the tree order being initial segment  $\sse$, and $T$ inherits the usual  lexicographic order from $2^{<\om}$.
The vertices of $\mathbb{H}_{n+1}$ appear as specific nodes in the tree $T$, $\hat{v}_j$ appearing at level $j$ of $T$.
Note that each subgraph $G\sse\mathbb{H}_{n+1}$ induces the  subtree of $T$ consisting of all  initial segments of the vertices in $G$.
In this appendix we shall use {\em vertex} to describe a  $\hat{v}_i$ in $\mathbb{H}_{n+1}$, and {\em node} to describe a node in the tree $T$.
When referring to a  vertex  $\hat{v}_i$ as a  node in $T$, we may also call it a {\em vertex node}.\footnote{In \cite{DobrinenJML20} and subsequent literature, such trees are called {\em coding trees}, and the   nodes in the tree corresponding the vertices $\hat{v}_j$ are called {\em coding nodes}.} 
Note that $\hat{v}_0$ is the empty sequence, and  $T$ has no terminal nodes. 
The tree $T$ 
has the following important property:
\begin{enumerate}
\item[$(*)$]
Given $k<\om$,
for each partition of 
$\{\hat{v}_i:i<k\}$ into two disjoint sets $V_0,V_1$  where $\mathbb{H}_{n+1}\upharpoonright V_1$ contains no $n$-clique,
there is a unique $s$ in $T$ of length $k$ so that for each $j\ge k$, 
$\hat{v}_j$ extends $s$ in $T$ iff 
$\hat{v}_i$ is disconnected with each vertex in $V_0$ and 
connected with every vertex in $V_1$.
\end{enumerate}

We point out that 
 $T$ is the tree of
all  (quantifier-free, consistent, complete)
$1$-types over finite initial segments of $\mathbb{H}_{n+1}$.
That is, for each $k<\om$, there is a 1-1 correspondence between the nodes in $T$ of length $k$ and the set of $1$-types over the finite graph $\mathbb{H}_{n+1}\upharpoonright\{\hat{v}_i:i< k\}$.
As such, 
$(*)$ can be restated model-theoretically:
For each $k<\om$, for each  $1$-type $\varphi(y; \hat{v}_0,\dots,\hat{v}_{k-1})$ over $\mathbb{H}_{n+1}\upharpoonright \{\hat{v}_i:i< k\}$ (where $y$ denotes a variable), there is a unique node $s\in T$ of length $k$ so that for each $j\ge k$, 
$\varphi(\hat{v}_j; \hat{v}_0,\dots,\hat{v}_{k-1})$ holds iff
 $\hat{v}_j$ extends $s$ in $T$.

The tree structure $(T, \sse)$ forms the bedrock for the notion of diary.
Splitting nodes in $T$ are the 
 second notion involved in diaries.
A  pair of vertices $\hat{v}_j,\hat{v}_k$  
 {\em splits  at level $\ell$} iff 
$\ell<\min(j,k)$ is minimal such that 
$\hat{v}_j(\ell)\ne \hat{v}_k(\ell)$.
In terms of the graph, this means that 
exactly one of $\hat{v}_j,\hat{v}_k$ is connected to $\hat{v}_\ell$, and
for each $i<\ell$, 
$\hat{v}_j$ is connected to $\hat{v}_i$ iff $\hat{v}_k$  is connected to $\hat{v}_i$.
We say that $\hat{v}_j,\hat{v}_k$  
{\em split}  iff  they split at level some level $\ell<\min(j,k)$.
Not every pair of vertices in $\mathbb{H}_{n+1}$  splits. There are many pairs $j<k$ so that for all $i<j$, $\hat{v}_k$ is connected to $\hat{v}_i$ iff $\hat{v}_j$ is connected to $\hat{v}_i$;
this corresponds to $\hat{v}_i$ being an initial segment of $\hat{v}_j$ in the tree $T$.
However, 
diaries will correspond only to subgraphs of $\mathbb{H}_{n+1}$ in which each pair of vertices splits. 
A node $s$ in a subset $S$ of $T$  is a {\em splitting node}  in $S$ iff both $s^{\frown}0$ and $s^{\frown}1$ are extended by some nodes in $S$.

The third notion involved in diaries is that of  consecutive age-changes.
A set $X\sse T$ is called a {\em level set} iff $X\sse 2^\ell$ for some $\ell<\om$.
Given a level set $X$, let  $\ell(X)$ denote the length of its nodes, and call it 
the {\em level} of $X$.
When we write a level set as $X=\{x_i:i<p\}$, it is assumed that the $x_i$ are all distinct, so $X$ has size $p$.
For $\ell<\ell(X)$, $X\upharpoonright\ell'$ denotes $\{x\upharpoonright\ell':x\in X\}$, the collection of nodes in $X$ truncated to  length $\ell'$.
This of course may have smaller size than $X$.

\begin{definition}\label{defn.Kmp}
For $1\le m< n$,
we define $K(m)$ to consist of those level sets 
 $X=\{x_i:i<p\}\sse T$ for which there is 
 an  $m$-clique  $\{\hat{v}_{k_q}:q<m\}$
where 
$k_0<\dots<k_{m-1}$, 
$\ell(X)=k_{m-1}+1$,  and 
for each $i<p$ and $q<m$, $x_i(k_q)=1$.
We say that such an  $m$-clique  {\em witnesses} that $X$ is in $K(m)$.

A level set $X=\{x_i:i<p\}$ is in $\AC(m,p)$ iff $X$ is in $K(m)$ and
for all $\ell'<\ell(X)$, $X\upharpoonright\ell'\not\in K(m)$.
We say that $X$ is an {\em age-change} iff $X\in \AC(m,p)$ for some $1\le m< n$  and  $1\le p  \le n+1-m$.
\end{definition}

We point out that age-changes do not occur for $p>n+1-m$ (see \cite{Balko7}).
If an $m$-clique  $\{\hat{v}_{k_q}:q<m\}$  witnesses that $X$ is in $K(m)$, then 
 any vertex  $\hat{v}_j$ 
 extending any node  in $X$ is connected to every member of the $m$-clique $\{\hat{v}_{k_q}:q<m\}$.
As $m$ increases, the possible connections between vertices in $\mathbb{H}_{n+1}$ extending the members of $X$ decreases. 
In particular, if $X$ is in $K(n-1)$, then 
any  collection of vertices in $\mathbb{H}_{n+1}$  extending members of $X$ must form an anticlique, since each such vertex forms an $n$-clique with $\{\hat{v}_{k_q}:q<n-1\}$.

 The notion of ``incremental coding trees" in 
 \cite{DobrinenJML20}  and \cite{DobrinenJML23}  led to the abstract notion of ``consecutive age-changes" for the larger class of  structures considered in  \cite{Balko7}.
 Here, we  give a constructive approach for  Henson graphs.

\begin{definition}[Consecutive age-change]\label{defn.conagechange}
For the following, 
$1\le m<n$ and 
 $1\le p  \le n+1-m$.
For $p=1$, for each $1\le m< n$ 
we define $\Con(m,1)=\AC(m,1)$;  
every  age-change  in $\AC(m,1)$ is {\em consecutive}.
For  $p\ge 2$,
we define $X\in \AC(m,p)$ to be a 
{\em consecutive age-change}, and write $X\in\Con(m,p)$,
 iff   
the following hold:
\begin{enumerate}
\item
For  each proper subset $Y\subset X$,
there is some (unique) $\ell<\ell(X)$ so that $Y\upharpoonright \ell$ is  
in $\Con(m,p')$, where $p'$ is the size  of $Y\upharpoonright \ell$.
\item
If $m\ge 2$, 
then for each $1\le m'<m$ and
 each $Y\sse X$ 
there is an $\ell<\ell(X)$ such that
 $Y\upharpoonright \ell$ is in $\Con(m',p')$, where $p'$ is the size of  $Y\upharpoonright \ell$.
\end{enumerate}
\end{definition}

To see why each age-change in $\AC(m,1)$ is  consecutive, note that 
  $X\in\AC(m,1)$
implies that
for each $1\le m'<m$ 
 there is an $\ell<\ell(X)$  such that $X\upharpoonright \ell$ is in $\AC(m',1)$; also, $X$
 being an age-change implies that  there is no
$\ell<\ell(X)$ so that  $X\upharpoonright \ell\in\AC(m,1)$.
Thus, all age-changes in $X$  for smaller $m'$ have already taken place at levels below $\ell(X)$.
It follows from our standard enumeration of $\mathbb{H}_{n+1}$ that  $\hat{v}_j$ is in $\Con(n-1,1)$, for each $j\ge n-1$.\footnote{This is called a {\em maximal path} in \cite{Balko7}.}

The fourth and final ingredient of diaries for Henson graphs involves so-called `controlled coding triples' in \cite{Balko7}.   
Here, we present an equivalent but more user-friendly version of this notion for Henson graphs.

\begin{definition}\label{defn.ccn}
Given a vertex $\hat{v}_j$ and a node $x\in T$ of length $j+1$ not extending $\hat{v}_j$, let  $1\le m<n$ be maximal such that $x\in K(m)$.
If  $x(j)=0$,
we say that  $x$ is {\em reduced   at $\hat{v}_j$} iff  for some $\ell<j$, the level set   $\{\hat{v}_j\upharpoonright \ell, x\upharpoonright \ell\}$ is in $K(m)$.
If  $x(j)=1$,
we say that  $x$ is {\em reduced   at $\hat{v}_j$} iff  for some $\ell<j$, the level set   $\{\hat{v}_j\upharpoonright \ell, x\upharpoonright \ell\}$ is in $K(m-1)$.
A level set $X$  consisting of nodes of length $j+1$ is  {\em reduced at
 $\hat{v}_j$}  iff 
 each $x\in X$   is  reduced
at $\hat{v}_j$.
\end{definition}

For example, in $\mathbb{H}_3$, 
suppose  $x$  is a node in the tree $T$  with length greater than $j$ and $x(j)=0$.
If  $x$ and $\hat{v}_j$
 are in different neighborhoods   then   $x$ is {\em reduced}  at $\hat{v}_j$ iff there is some $i<j$ such that $x(i)=\hat{v}_j(i)=1$; that is,
the pair $\{\hat{v}_j, x\upharpoonright j\}$ is  a member of $K(2)$.
If  $x$ and $\hat{v}_j$ are in the same neighborhood, then  they already have a common initial segment in $K(2)$.

%{\color{red}  The controlled coding triples are a beast and it is not in our best interests to define them here since there are many of them.  I think the best approach is to just say that there are finitely many things to add just before the levels of the terminal coding nodes and that we can do so computably.  Also, note that the diary I construct in my arxiv paper does exactly what we want in this paper.}

\subsection{Diaries}
We give a constructive description  of   the expanded structures that  exactly characterize
the big Ramsey degrees  in $\mathbb{H}_{n+1}$, $n\ge 2$.
The terminology  has  presently converged to call these expanded structures
{\em diaries}.
(Previous  terminology for diaries of various infinite structures includes {\em canonical partitions}, {\em similarity types}, and {\em embedding types}.)
%The diary of a subgraph $G$ of $\mathbb{H}_{n+1}$ is induced by    particular information about how $G$ are connected to nodes in $\mathbb{H}_{n+1}$.
%Not every subgraph of $\mathbb{H}_{n+1}$ will induce a diary. But every diary will correspond to infinitely many subgraphs of $\mathbb{H}_{n+1}$.  Importantly,  we will make a particular copy of $\mathbb{H}_{n+1}$ that is  itself a diary, and hence, all of its subgraphs will induce  diaries. 

Simple descriptions of diaries for $\mathbb{H}_3$ were given in \cite{DobrinenH_3ExactDegrees20} and  \cite{Dobrinen_ICM}.
Diaries were defined abstractly in Definition 3.4.1 in
\cite{Balko7} for all 
universal ultrahomogeneous structures obtained as \Fraisse\ limits of 
free amalgamation classes of finite structures in finitely many binary relations with finitely many   forbidden  finite, irreducible substructures.
%Example 3.4.4 in \cite{Balko7} gives a description  of diaries for $\mathbb{H}_3$ in the terminology of that paper; however, while that description simplifies Definition 3.4.1 in \cite{Balko7}, one still has to plod  through much notation before gaining an understanding.
%Recently,a simple description of diaries for $\mathbb{H}_4$ is given in \cite{new paper}. (As I had found, their description has some mistakes.)
The  description of 
diaries in \cite{Balko7} 
include bounded  versions of 
the notions
of  ``incremental" coding trees (Definition 9.1, \cite{DobrinenJML23})
and   the notion of ``strict similarity type"  (Definition 9.4, \cite{DobrinenJML23})
developed in \cite{DobrinenJML20} for $\mathbb{H}_3$ and in \cite{DobrinenJML23} for all $\mathbb{H}_{n+1}$.
Two  modifications of these notions, namely,  the bound $p$ in Definition \ref{defn.Kmp}  that depends on both $n$ and $m$, 
and  the addition of so-called  ``controlled coding triples", which we construct via reduced pairs here,
determine diaries.
We say that a node in $T$ is in the {\em spine} of $T$ if it is a finite sequence of $0$'s.

We now define diaries for Henson graphs. 

%describe the modifications of  work in \cite{DobrinenJML23}  that  produce diaries for the graphs $\mathbb{H}_{n+1}$, giving a more user-friendly approach than in \cite{Balko7}.

\begin{definition}\label{def.diary}
Let $G=\{\hat{v}_{j_i}:i<g\}$, ($1\le g\le \om$),  be  a subgraph of $\mathbb{H}_{n+1}$.  
We say that  $G$  {\em induces  a   diary}  iff 
letting $A=A(G)$ denote the meet-closure  in $T$ of the vertices $\{\hat{v}_{j_i}:i<g\}$,
the following hold:
\begin{enumerate}
\item
Each pair of  vertices  in $G$ splits.
Hence,  the vertices in $G$  are exactly the terminal nodes in $A$.
\item 
Each age-change in $A$ is consecutive.
\item
Each level of $A$ has at most one of three types of 
events:
\begin{enumerate}
\item[(a)]
Exactly one splitting node.
This can either be the meet of a vertex node with the spine, or a splitting node that is the meet of two vertex nodes in $A$.
Moreover, except for the one node that extends the splitting node by $1$, all other nodes at this level extend by $0$ (ensuring there are no age-changes, except  when the splitting node  is splitting off from the spine). 
\item[(b)]
Exactly one  consecutive age-change. In this case, the set of nodes involved in the consecutive age-change and the  type of age-change, $\Con(m,p)$,   are recorded.
\item[(c)]
A vertex, say $\hat{v}_{j_i}$. In this case, for each $i<k<g$, the value of $\hat{v}_{j_k}(|\hat{v}_{j_i}|)$ is recorded.
\end{enumerate}
The {\em interesting levels} of $A$
are those levels 
in which an event of {\em type} (a), (b), or (c) takes place.
\item
In (3c), if $\hat{v}_{j_i}$ is the vertex in a given interesting level, then 
letting $\ell=|\hat{v}_{j_i}|$,
each node $x$ in $A$ of length greater than $\ell$ with $x(\ell)=0$ is reduced.
These reductions show up as consecutive age-changes as part of (3b).
\end{enumerate}

 Let    $L(A)$ denote the set of those $\ell$ for which $A\upharpoonright\ell$ is an  interesting level of $A$, and let $L=|L(A)|$.
If  $G$  induces a diary, then  the {\em diary}
 of $G$,  $\Delta=\Delta(G)$,  is defined to be 
the subtree of $2^{< L}$ 
that is isomorphic (in both tree and lexicographic orders) to
$\bigcup_{\ell\in L(A)} A\upharpoonright\ell$,
 along with  the following information:
 the type of each level of interest,
 the node(s) in the level  involved in the type,
 and in the case of a vertex node,
 the  value ($0$ or $1$) of the immediate successors of all non-vertex nodes in that level
 (equivalently, the edge/non-edge relations between the vertices in $G$).
Two graphs $F$ and $G$ are said to {\em have the same diary} 
iff $\Delta(F)$ and $\Delta(G)$  are equal.
\end{definition}

Notice that the number of diaries for a given finite graph $G$ is finite.
Also, if $G\sse T$ induces a diary, then every subgraph of $G$ also induces a diary. 
We now (computably in $\mathbb{H}_{n+1}$) construct 
a subtree $\mathbb{A}$ of $T$ that will induce a diary $\mathbb{D}_{n+1}$  whose vertex nodes form
an order-isomorphic copy $\widehat{\mathbb{H}}_{n+1}$.
Our construction closely follows  constructions in  \cite{DobrinenJML20} and \cite{DobrinenJML23}.
%the modifications being the simplification of the bound on $p$ and the addition of reduced nodes. 
A general construction of such diaries for free amalgamation classes with finitely many binary and unary relations, with finitely many forbidden finite irreducible substructures,  is found in \cite{Dobrinen/Zucker23}.
Here, we provide a concrete 
construction   for  Henson graphs.

\begin{definition}\label{defn.ai}
Given the tree $T$ induced by $\mathbb{H}_{n+1}$, for a subtree $S\sse T$ with vertex nodes $\{v_i:i<N\}$, where $N\le \om$,
let  $\Lev_S(v_i)$ denote the set of nodes in $S$ of length $|v_i|$,
and let
$\Lev^+_S(v_i)$ denote the set of nodes in $S$ of length $|v_i|+1$.
We say that a subtree $S\sse T$ is {\em essentially isomorphic} to $T$ iff the following hold:
\begin{enumerate}
\item[(a)]
The vertex nodes in $S$ are  $\{v_i:i<\om\}$, where $v_i=\hat{v}_{j_i}$ for some $j_i$, and $\mathbb{H}_{n+1}\upharpoonright\{v_i:i<\om\}$ is order-isomorphic to $\mathbb{H}_{n+1}$.
\item[(b)]
The vertex nodes in $S$ are exactly the terminal nodes in $S$.
\item[(c)]
For each $i<\om$,
$\Lev_S^+(v_i)$ has the same 
the number of nodes  as $\Lev_T^+(\hat{v}_i)$.
\item[(d)]
For each $i<\om$, 
letting $f$ be the lexicographic order-preserving map from 
$\Lev_T^+(\hat{v}_i)$ to
$\Lev_S^+(v_i)$,
 for each $x\in \Lev_T^+(\hat{v}_i)$ and each $j\le i$,
$(f(x))(|v_j|)=x(|\hat{v}_j|)$.
\end{enumerate}
\end{definition}
Note that 
(d) says that the $1$-type of $x$ over $\mathbb{H}_{n+1}\upharpoonright\{\hat{v}_j:j\le i\}$ is the same as the $1$-type of $f(x)$ over $\mathbb{H}_{n+1}\upharpoonright\{v_j:j\le i\}$, upon substituting $v_j$ for $\hat{v}_j$.
Hence,
(c) and (d) imply (a). 
Further, (d) implies  that for each  $1$-type over  $\mathbb{H}_{n+1}\upharpoonright\{v_j:j\le i\}$, there is exactly one node in  $\Lev_S^+(v_i)$
  with that $1$-type.
In particular, this implies that each node in $\Lev_S^+(v_i)$ has infinitely many vertices extending it (by the Extension Property of $\mathbb{H}_{n+1}$). 
Next  is the main theorem of the Appendix.

\begin{theorem}\label{thm.diary}
Given $\mathbb{H}_{n+1}$,
there is a   subtree $\mathbb{A}\sse T$, computable in $\mathbb{H}_{n+1}$,
such that the following hold:
\begin{enumerate}
\item 
$\mathbb{A}$ is  computable in $\mathbb{H}_{n+1}$ and is essentially isomorphic to $T$.
\item
$\mathbb{A}$ induces a diary, denoted by $\mathbb{D}_{n+1}$, which is computable in $\mathbb{H}_{n+1}$.
\item 
The vertex nodes in $\mathbb{D}_{n+1}$
are the same as the vertex nodes in
$\mathbb{A}$, and  
these  induce an order-isomorphic copy $\widehat{\mathbb{H}}_{n+1}$ of 
$\mathbb{H}_{n+1}$.
\item
Every collection of vertices in 
$\widehat{\mathbb{H}}_{n+1}$ induces a diary.
\item 
Let $G$ be a
 finite $(n+1)$-clique free graph and  $\Delta$ be   a diary for $G$.
The set of all $(|G|+1)$-tuples $\{m,j_0,j_1,\dots,j_{|G|-1}\}$ 
of vertices 
$\{\hat{v}_{j_0},\dots,\hat{v}_{j_{|G|-1}}\}$ from copies of $G$ in 
$\widehat{\mathbb{H}}_{n+1}$,
where $\hat{v}_{j_0}=\hat{v}_{m,j_0}$,
with induced diary $\Delta$  is largely $2$-extendable.
\item 
Let $G$ be a
 finite $(n+1)$-clique free graph and  $\Delta$ be   a diary for $G$.
 For each copy $G'$ of $G$ in  $\widehat{\mathbb{H}}_{n+1}$ with induced diary  $\Delta$, let $L(G')$ denote the set of interesting  levels  of $A(G')$.
  Then the set of these $L(G')$ is
 largely $2$-extendable.
\end{enumerate}
\end{theorem}

\begin{proof}
Without loss of generality, assume that $\mathbb{H}_{n+1}$ is  computable with a standard enumeration.
Let $T$ denote the tree 
induced by the vertices in $\mathbb{H}_{n+1}$, noting that $T$ is computable in $\mathbb{H}_{n+1}$.
Recall that in the tree $T$, $\hat{v}_i$ is a sequence  in $2^i$, and for
 a node $s\in T$, $|s|$ denotes the length of $s$.
We will, computably in $\mathbb{H}_{n+1}$,  build a subtree $\mathbb{A}\sse T$ satisfying (1)--(6).

We  now give an overview of our construction of $\mathbb{A}$.
We will choose vertex nodes  $v_i=\hat{v}_{j_i}$  in $T$ so that $\mathbb{H}_{n+1}\upharpoonright\{v_i:i<\om\}$ is order-isomorphic to $\mathbb{H}_{n+1}$.
These will be the terminal nodes in $\mathbb{A}$.
For each $1\le i<\om$,
between the levels of $v_{i-1}$ and $v_i$  in $\mathbb{A}$,  there will  be three intervals of levels, each consisting of a different type of level.
The first is  a {\em splitting interval}.
Such an interval 
will contain levels that each contain exactly one
 splitting node, those of type (3a) in Definition \ref{def.diary}.
Letting $n(i)$ denote the number of nodes in 
$\Lev_T^+(\hat{v}_i)$,
there will be exactly $n(i)+1$ many splitting nodes in the splitting interval just preceding $v_i$.
Since these are the only splitting nodes between  $v_{i-1}$ and $v_i$ (or below $v_i$ if $i=0$), and since $v_i$ will be terminal in $\mathbb{A}$, it will follow that $\Lev_{\mathbb{A}}^+(v_i)$ has exactly $n(i)$ many nodes. 
The second type is a {\em consecutive age-change interval}.
The number of levels in this interval and the consecutive age-changes in it are completely determined by the age-changes that occur in $T$ at $\Lev_T^+(\hat{v}_i)$.
The third type is a
{\em reducing interval}.
Such an 
interval will contain any further consecutive age-changes necessary to reduce,  with $v_i$, each $s\in \Lev_{\mathbb{A}}^+(i)$ satisfying  $s(|v_i|)=0$.
(These are determined by $T$ up to the level of $\hat{v}_{i+1}$.)
The level of $\mathbb{A}$  containing $v_i$ will be minimal  above this reducing interval. 
In particular, 
$\mathbb{A}$ will satisfy the requirements in Definition \ref{def.diary} and thus induce a diary, which we denote by $\mathbb{D}_{n+1}$.

The construction of $\mathbb{A}$ begins with 
declaring $\Lev_{\mathbb{A}}(0)$ to consist of 
the root of  $T$, namely the empty sequence. This node will be  
 a splitting node in $\mathbb{A}$.
The tree $T$ has exactly two nodes 
in $\Lev_T^+(0)$, so 
let $n(0)=2$.
We will need  two nodes in $\mathbb{A}$ at the level of  $v_0$ in addition to $v_0$, so we need one more splitting node before choosing $v_0$.
Let $s_1$ be the shortest  splitting node  in $T$ extending  $\langle 1\rangle$, and let $s_0$ be the sequence of $0$'s of length $|s_1|$.
Let  $\Lev_{\mathbb{A}}(1)=\{s_0, s_1\}$.
This  $s_1$ is  the splitting node in $\Lev_{\mathbb{A}}(1)$.
This finishes the splitting interval below $v_0$.
In $T$, 
the only age-change 
in $\Lev_T^+(\hat{v}_0)$ is the node $\langle 1\rangle$, and it is consecutive and  is  already reduced with $\hat{v}_0$.  Thus, it just remains to choose $v_0$.
Let $j_0$ be least such that $\hat{v}_{j_0}$ extends 
${s_1}^{\frown}0$, and let $v_0=\hat{v}_{j_0}$.
Let $t_0$ be the sequence of  $0$'s of length $|v_0|$ and $t_1$ be the sequence of length $|v_0|$ that extends ${s_1}^{\frown}1$ by $0$'s.
(The extension of any node in $T$ by finitely many $0$'s is always again a node in $T$.)
Let  $\Lev_{\mathbb{A}}(2)=\{v_0,t_0,t_1\}$. 
Note that $\mathbb{A}$ up to level $2$ induces a diary. 
Let 
  $\Lev^+_{\mathbb{A}}(v_0)=
\{{t_0}^{\frown}0,{t_1}^{\frown}1\}$.
The levels of $\mathbb{A}$ we have constructed so far satisfy Definitions \ref{def.diary} and \ref{defn.ai} as well as the inductive hypotheses below.

Suppose now that $i\ge 1$ and we have 
 constructed $\mathbb{A}$ up to
 $\Lev_{\mathbb{A}}^+(v_{i-1})$
satisfying Definitions \ref{def.diary} and \ref{defn.ai}
  so that the following inductive hypotheses hold:
\begin{enumerate}
\item
For all $j<i$, 
$\Lev_\mathbb{A}(v_j)\setminus\{v_j\}$ and $\Lev_\mathbb{A}^+(v_j)$ each have  exactly $n(j)$ nodes, where $n(j)$ is the number of nodes in $\Lev_T^+(\hat{v}_j)$.
\item
Enumerating  $\Lev_\mathbb{A}^+(v_j)$ and $\Lev_T^+(\hat{v}_j)$ as $\{x_m:m<n(j)\}$ and $\{t_m:m<n(j)\}$,  respectively, in lexicographic order,
then for each $m<n(j)$, $x_m(|v_j|)=t_m(j)$ $(=t_m(|\hat{v}_j|))$.
%Moreover, $x_0$ is a splitting node in $T$.
\item
For each $P\sse n(j)$ of size at most $n$, $\{x_k:k\in P\}$ 
is in  $K(m)$
iff 
$\{t_k:k\in P\}$ is in 
 $K(m)$,
 where $2\le m\le n$.
\end{enumerate}
Note that (2) implies that  $\mathbb{H}_{n+1}\upharpoonright\{v_j:j<i\}$ is order-isomorphic to $\mathbb{H}_{n+1}\upharpoonright\{\hat{v}_j:j<i\}$.
(3) implies that we can continue building $\mathbb{A}$ above $v_j$ to make an almost isomorphic copy of $T$.

We now do the inductive step of constructing the splitting, consecutive age-change, and reducing intervals above $v_{i-1}$ and choosing $v_i$ and $\Lev_{\mathbb{A}}^+(v_i)$.
Let $X=\{x_m:m<n(i-1)\}$ list the nodes in $\Lev_\mathbb{A}^+(v_{i-1})$ in lexicographic order,
and let $Y=\{y_m:m<n(i-1)\}$ list the nodes in  $\Lev_T^+(v_{i-1})$ in lexicographic order.

We first construct the splitting interval above $v_{i-1}$.
Let $M=\{m_0,\dots,m_k\}$ be the set of those indices $m<n(i-1)$ such that $y_m$
is a splitting node in $T$.
Note that    $y_0$ and $x_0$ are both   sequences of $0$'s, so are they are both splitting nodes in $T$ and hence, $m_0=0$.
Let $s_{m_0}$ denote $x_0$. 
Successively choose splitting nodes $s_{m_0}, s_{m_1},\dots, s_{m_k}$ extending $x_0, x_{m_1},\dots, x_{m_k}$, respectively, 
so that $|s_0|<|s_{m_1}|<\dots<|s_{m_k}|$.
For each $j\le k$, let $s'_{m_j,0}$ and $s'_{m_j,1}$ be the nodes of length $|s_{m_k}|+1$ extending ${s_{m_j}}^{\frown}0$ and ${s_{m_j}}^{\frown}1$, respectively. 
For each $m\in n(i-1)\setminus M$,  let $s'_m$ be the node with length $|s_{m_k}|+1$ extending $x_m$ by $0$'s.
Let 
$$
X'=\{s'_{m_j,0}, s'_{m_j,1}:j\le k\}\cup \{s'_m:m\in n(i-1)\setminus M\}.
$$
Then $X'$ is a level set of nodes, extending the nodes in $X$, with the same size as $\Lev_T^+(\hat{v}_i)$, that is, $n(i)$.

Let $Z=\{z_m:m<n(i)\}$ enumerate the nodes in $\Lev_T^+(\hat{v}_i)$ in lexicographic order, and 
let  $\tilde{m}$ be the index  where $z_{\tilde{m}}=\hat{v}_i$.
List the nodes in $X'$ in lexicographic order as $\{x'_m:m<n(i)\}$. 
Extend $x'_{\tilde{m}}$ to a splitting node $s$ in $T$.
Let $x''_{\tilde{m},0}={x'_{\tilde{m}}}^{\frown}0$ and 
$x''_{\tilde{m},1}={x'_{\tilde{m}}}^{\frown}1$, and extend all other nodes in $X'$ by $0$'s to the length of $|x'_{\tilde{m}}|+1$.
Label this level set $X_0$.
Note that $X_0$ has $n(i)+1$ many nodes.  The node $x''_{\tilde{m},0}$ will be extended to a vertex node $v_i$ after we construct  the reducing interval. 
This concludes the construction of the splitting interval.

Second, we construct the   consecutive age-change interval.
The following is a  canonical way  to list, computably in $\mathbb{H}_{n+1}$, all consecutive age-changes that need to occur between $v_{i-1}$ and $v_i$ in $\mathbb{A}$.
List  all subsets $U$ of $\Lev_T^+(\hat{v}_i)$ which are 
in $\AC(m,p)$ for some
 $2\le m<n$ and  $1\le p
\le n+1-m$
and such that $U\not\in \AC(m+1,p)$, and  moreover,   $U$ being maximal in the sense that  if $p<n+1-m$ then for any other $y\in \Lev_T^+(\hat{v}_i)\setminus U$, $U\cup\{y\}$ is not in $\AC(m,p+1)$.
We extend the nodes in $X_0$ successively so that for each such $U$, 
we add consecutive age-changes for all subsets of $U$ 
(in order of size of the subsets of $U$) and extend all other nodes by $0$'s.\footnote{This is a byproduct of the constructions in \cite{DobrinenJML20} and \cite{DobrinenJML23}, and we refer the reader to these papers rather than reproduce the construction here.}
At the end, we have a level set of nodes $X_1$ of the same cardinality,  $n(i)+1$, as $X_0$ satisfying (3) of the induction hypothesis.
This is procedure is computable in $\mathbb{H}_{n+1}$.

Thirdly, we construct the reducing interval. 
Recall that there are $n(i)+1$ many nodes in  $X_1$.
Let $\tilde{x}$ denote the node in $X_1$ extending $x''_{\tilde{m},0}$;
thus, $\tilde{x}$ is the node that will be extended to $v_i$.
List the nodes in $X_1\setminus\{\tilde{x}\}$ as $\{x^1_m:m<n(i)\}$ in lexicographic order, and list the nodes in $\Lev^+_T(\hat{v}_i)$
in lexicographic order as $\{t_m:m<n(i)\}$.
Let $N$ be the set of those indices $m<n(i)$, $m\ne 0$,  for which $t_m(|\hat{v}_i|)=0$.
For $m_0$ least in $N$,
extend the pair 
 $x^1_{m_0},\tilde{x}$ to a reduced pair $x^2_{m_0}, \tilde{x}_{m_0}$.
 For  the next least $m_1$ in $N$, 
 extend $x^1_{m_1}$ by $0$'s to the length of $\tilde{x}_{m_0}$ and then 
 extend that pair  to a reduced pair 
 $x^2_{m_1},\tilde{x}_{m_1}$.  Note that $|\tilde{x}|<|\tilde{x}_{m_0}|<|\tilde{x}_{m_1}|<\dots$.
 For $m'=\max(N)$, extend $\tilde{x}_{m'}$ to a vertex node $v_i$.
 For each $m\in N$, let $x^3_m$ be the 
 extension of   $x^2_{m}$ by $0$'s to the length of $v_i$.
 For each node  $x_m^1$ in $X_1$ not considered yet, let $x^3_m$ be the extension of $x_m^1$ by $0$'s to the length of $v_i$.
Then let $\Lev_{\mathbb{A}}(v_i)=\{x^3_m:m<n(i)\}\cup\{v_i\}$.

 Lastly, for each $m<n(i)$, let $k_m$ be the member of $\{0,1\}$  so that 
 $z_m(|\hat{v}_i|)=k_m$
and let 
$$
\Lev^+_{\mathbb{A}}(v_i)=\{{x^3_m}^{\frown}k_m:m<n(i)\}.
$$
Our construction ensures that the induction hypotheses (1)--(3)  hold at this stage of the construction.

Let $\mathbb{A}$ be the closure under initial segments of $\bigcup_{i<\om}\Lev_{\mathbb{A}}(v_i)$.
Then $\mathbb{A}$ is computable in $\mathbb{H}_{n+1}$,
the vertex nodes $v_i$ of $\mathbb{A}$ form an antichain and are exactly the terminal nodes in $\mathbb{A}$, and $\mathbb{A}$ is almost isomorphic to $T$; so (1)  of the theorem holds.
The collection of the splitting node levels, consecutive age-change levels, and vertex node levels of $\mathbb{A}$ form a diary, call it $\mathbb{D}_{n+1}$.
Moreover, these levels are determined by $\mathbb{H}_{n+1}$ and $\mathbb{A}$, so $\mathbb{D}_{n+1}$ is computable in $\mathbb{H}_{n+1}$ and hence  (2) holds. 
(3) follows from the definition of $\mathbb{D}_{n+1}$ and the fact that $\mathbb{A}$ is almost isomorphic to $\mathbb{H}_{n+1}$.
If $C$ is a collection of  vertices in $\widehat{\mathbb{H}}_{n+1}$,
then $C$ is a collection of vertices in $\mathbb{D}_{n+1}$. The downward closure of $C$ in $\mathbb{D}_{n+1}$ is a diary.
Thus, (1)--(4) of the theorem hold.

We now prove (6), from which (5) directly follows.
Fix  $G$, a finite $(n+1)$-clique free graph in $\widehat{\mathbb{H}}_{n+1}$, and let $\Delta$ denote its induced diary.
Let 
$g=|G|$ and let
 $\tilde{k}$ be the number of levels in $\Delta$.
Let $\mathcal{G}$ be the collection of all copies $G'$ of $G$ in $\widehat{\mathbb{H}}_{n+1}$  for which $\Delta(G')$ equals $\Delta$.
Let $\mathcal{F}$ be the set of
all $\tilde{k}$-tuples $L(G')$, where  $G'\in\mathcal{G}$.
We will show that $\mathcal{F}$ is largely $2$-extendable.
Recall that since  $\mathbb{H}_{n+1}$, and hence $\widehat{\mathbb{H}}_{n+1}$, has a standard enumeration, every vertex node in $\mathbb{A}$ is off the spine.
It follows that for each $G'\in \mathcal{G}$, the least level of 
$\Delta(G')$ consists of a splitting node  on the spine of $\mathbb{A}$.
By Theorem \ref{local1},  
there are infinitely many splitting nodes  in the spine of $\mathbb{A}$.
Hence, $\mathcal{F}_1$ is infinite, so 
 (1) of Definition \ref{def.l2e} holds.

Now fix any $\ell_0\in\mathcal{F}_1$, and fix some
 $G'\in\mathcal{G}$ with $\ell_0$  least in $L(G')$.
 Then $\ell_0$ is 
 the length of 
 the  least node  $s_0$ in the meet-closure $A$ of $G'$.
Let $\ell_1$  be the second member of $L(G')$, and let $X=A\upharpoonright\ell_1$.
Note that $X$ can have one or two nodes. 
$X$ must have 
 a node off the spine, call it $x_1$,
and $X$ may or may not have a node on the spine, call it $x_0$ if there is one. 
Note that $x_1$ extends ${s_0}^{\frown}1$.

If $x_1$ is a vertex node, then it follows from $\mathbb{A}$ being essentially isomorphic to $T$ that there are infinitely many vertex nodes in $\mathbb{A}$  extending ${s_0}^{\frown}1$.
If $x_1$ is a consecutive age-change, then it  is in $\Con(2,1)$, since ${s_0}^{\frown}1$ is already in $\Con(1,1)$.
Since $\mathbb{A}$ essentially isomorphic to $T$,
there are infinitely many nodes extending ${s_0}^{\frown}1$ in $\Con(2,1)$.
If $x_1$ is a splitting node,
then again,  since $\mathbb{A}$ is essentially isomorphic to $T$,
there are infinitely many splitting nodes extending ${s_0}^{\frown}1$.
All of the above cases work for $X$ of size $1$ or $2$.
Lastly, if $X$ has size $2$ and $x_0$ is a splitting node, 
then there are infinitely many splitting nodes on the spine extending ${s_0}^{\frown}0$.
In  all cases, 
there are infinitely many $\ell_1$ for which $(\ell_0,\ell_1)$ is in $\mathcal{F}_2$.
Thus,  (2) of Definition \ref{def.l2e} holds.  In fact, we proved more: for each $\ell_0\in\mathcal{F}_1$, the set of $\ell_1>\ell_0$ for which $(\ell_0,\ell_1)$ is in $\mathcal{F}_2$ is infinite.

The rest of the proof shows
that  (3) of Definition \ref{def.l2e} holds, by induction on 
$k<\tilde{k}$.  Suppose $(\tilde{\ell}_0,\tilde{\ell}_1,\dots, \tilde{\ell}_{\tilde{k}-1})\in\mathcal{F}$ and let $2\le k< \tilde{k}$.
We will show that  there are infinitely many  $\ell_k$ such that  
such that $(\tilde{\ell}_0,\dots, \tilde{\ell}_{k-1},\ell_k)\in\mathcal{F}_{k+1}$.
Let $G'\in\mathcal{G}$ be such that $L(G')=(\tilde{\ell}_0,\dots, \tilde{\ell}_{k-1},\ell_k)$.
Let $A$ be the meet-closure of $G'$ in $\mathbb{A}$,  let $X= A\upharpoonright \tilde{\ell}_{k}$, and let $A'$ be the truncation of $A$ to the level of $\tilde{\ell}_{k-1}$.

If $X$ is an interesting level of  type (a), 
then, letting $x$ denote the splitting node in $X$,  there are infinitely many splitting nodes $s$ in $\mathbb{A}$  extending $x$.
Extending the other nodes in $X$ leftmost in $\mathbb{A}$ to the level of $s$ makes another interesting level $Y$ of type (a) extending $A'$.
Such  $A'\cup Y$ can be extended to some member of $\mathcal{G}$, so there are infinitely many $\ell$ for which 
$(\tilde{\ell}_0,\dots, \tilde{\ell}_{k-1},\ell_k)$ is in $\mathcal{F}_{k+1}$.

Suppose that  $X$ is an interesting level of  type (b); that is, 
 $X$ is a consecutive age-change above $A'$.
For each $\ell'>\tilde{\ell}_{k-1}$, any leftmost extension of $X\upharpoonright(\tilde{\ell}_{k-1}+1)$ to length $\ell'$ has no age-change over $A'$; so it can be extended to a consecutive age-change over $A'$.
Hence, there are infinitely many 
$\ell$ for which 
$(\tilde{\ell}_0,\dots, \tilde{\ell}_{k-1},\ell_k)$ is in $\mathcal{F}_{k+1}$.

Lastly, suppose $X$ is an interesting level of  type (c);
that is, $X$ contains a vertex node, say $v_i$.
Let $x=v_i\upharpoonright(\tilde{\ell}_{k-1}+1)$.
For each $\ell'>\tilde{\ell}_{k-1}$, there is a vertex node $v_j$ extending $x$ of length at least $\ell'$. 
Extending all other nodes in 
$\Lev^+_{A}(\tilde{\ell}_{k-1})$ leftmost to the length of $v_i$ produces an interesting level of type (c) which can be extended to a graph in $\mathcal{G}$.
Thus, there are infinitely many $\ell$ for which 
$(\tilde{\ell}_0,\dots, \tilde{\ell}_{k-1},\ell_k)$ is in $\mathcal{F}_{k+1}$.

Thus, in all three cases, we see that $\mathcal{F}_{k+1}$ is largely $2$-extendable.  By induction on $k$, we see that $\mathcal{F}$ is largely $2$-extendable.
Part (5) of this theorem follows immediately from (6), since the tuples $\{m,j_0,\dots,j_{|G|-1}\}$
are subsets of $L(G')$ consisting of the minimum member of $L(G')$ and the lengths of the vertices in $G'$.
\end{proof}

%{\color{red}  I could make a lemma that shows or at least states how to computably in $\mathbb{H}_{n+1}$ make consecutive age-changes. This was done in \cite{DobrinenJML23}. }

 Theorem \ref{lem.D} is (6) of Theorem \ref{thm.diary}.  Theorems  \ref{dairies} follows below from the previous theorem, Theorem~\ref{thm.diary}.

\medskip

\noindent {\it Proof of Theorem \ref{dairies}.}
The fact that there is a computable function $k(G,n)$ such that there are $k(G,n)$ diaries corresponding to copies of $G$ in $\mathbb{H}_{n+1}$ follows from the definition of diary.  This was known by work in \cite{Balko7}, and can be seen from work in \cite{DobrinenJML23}.
That there is a  copy $\widehat{\mathbb{H}}_{n+1}$, computable in $\mathbb{H}_{n+1}$,  so that 
each copy of $G$ in $\widehat{\mathbb{H}}_{n+1}$, and hence in any subcopy $\tilde{\mathbb{H}}$ of $\widehat{\mathbb{H}}_{n+1}$, corresponds to exactly one diary, is a restatement of (1)--(4) in Theorem \ref{thm.diary}.
The constructions  of $\mathbb{A}$ and $\mathbb{D}_{n+1}$ being computable shows that 
there is a computable function $D(G,\tilde{\mathbb{H}})$ which outputs this diary  (using $\tilde{\mathbb{H}}$ as an oracle). 
That diaries for close copies of $G$ are different from the diaries for copies of $G$ which are not close follows from the definition of diary.  \hfill $\square$
\medskip

\bibliographystyle{plainnat}
\bibliography{H3b.bib}

@misc{Dobrinen/Zucker23,
	archiveprefix = {arXiv},
	author = {Natasha Dobrinen and Andy Zucker},
	eprint = {2303.04246},
	primaryclass = {math.LO},
	title = {Infinite-dimensional {R}amsey theory for binary free amalgamation classes},
	url = {https://arxiv.org/abs/2303.04246},
	year = {2023}}

@misc{DobrinenH_3ExactDegrees20,
	archiveprefix = {arXiv},
	author = {Natasha Dobrinen},
	eprint = {2009.01985},
	primaryclass = {math.CO},
	title = {The {R}amsey theory of the universal homogeneous triangle-free graph {P}art {II}: {E}xact big {R}amsey degrees},
	url = {https://arxiv.org/abs/2009.01985},
	year = {2020}}

@article{J72,
	author = {Jockusch, Jr., Carl G.},
	fjournal = {The Journal of Symbolic Logic},
	journal = {J. Symbolic Logic},
	pages = {268--280},
	title = {Ramsey's theorem and recursion theory},
	volume = {37},
	year = {1972}}

@article{MR3518779,
	author = {Hirschfeldt, Denis R. and Jockusch, Jr., Carl G.},
	doi = {10.1142/S0219061316500021},
	fjournal = {Journal of Mathematical Logic},
	issn = {0219-0613,1793-6691},
	journal = {J. Math. Log.},
	mrclass = {03D30 (03B30 03D80 03F35)},
	mrnumber = {3518779},
	mrreviewer = {Ludovic\ Patey},
	number = {1},
	pages = {1650002, 59},
	title = {On notions of computability-theoretic reduction between {$\Pi_2^1$} principles},
	url = {https://doi.org/10.1142/S0219061316500021},
	volume = {16},
	year = {2016}}

@article{DobrinenJML20,
	author = {Dobrinen, Natasha},
	journal = {J. Math. Log.},
	number = {2},
	pages = {Paper No. 2050012, 75 pp},
	title = {The {R}amsey theory of the universal homogeneous triangle-free graph},
	volume = {20},
	year = {2020}}

@article{DobrinenJML23,
	author = {Dobrinen, Natasha},
	journal = {J. Math. Log.},
	number = {1},
	pages = {Paper No. 2250018, 88 pp.},
	title = {The {R}amsey theory of {H}enson graphs},
	volume = {23},
	year = {2023}}

@article{Dobrinen_ICM,
	author = {Dobrinen, Natasha},
	journal = {ICM---International Congress of Mathematicians, Vol.\ 3, Sections 1--4},
	pages = {1462--1486},
	title = {Ramsey theory of homogeneous structures: current trends and open problems},
	year = {2023}}

@article{El-Zahar/Sauer89,
	author = {El-Zahar, Mohamed and Sauer, Norbert},
	journal = {Journal of Combinatorial Theory, Series B},
	number = {2},
	pages = {162--170},
	title = {The indivisibility of the homogeneous ${K}_n$-free graphs},
	volume = {47},
	year = {1989}}

@article{Komjath/Rodl86,
	author = {Komj{\'{a}}th, P{\'{e}}ter and R{\"{o}}dl, Vojt{\v{e}}ch},
	journal = {Graphs and Combinatorics},
	number = {1},
	pages = {55--60},
	title = {Coloring of universal graphs},
	volume = {2},
	year = {1986}}

@article{Balko7,
	author = {Balko, M. and Chodounsk\'y, D. and Dobrinen, N. and Hubi\v{c}ka, J. and Kone\v{c}n\'y, M. and Vena, L. and Zucker, A.},
	journal = {J.\ Eur.\ Math.\ Soc.},
	number = {5},
	pages = {2101--2150},
	title = {Exact big {R}amsey degrees for finitely constrained binary free amalgamation classes},
	volume = {28},
	year = {2026}}

@misc{Gill23,
	archiveprefix = {arXiv},
	author = {Kenneth Gill},
	eprint = {2310.20097},
	primaryclass = {math.LO},
	title = {A note on the indivisibility of the {H}enson graphs},
	url = {https://arxiv.org/abs/2310.20097},
	year = {2023}}

@article{MR268080,
	author = {Folkman, Jon},
	doi = {10.1137/0118004},
	fjournal = {SIAM Journal on Applied Mathematics},
	issn = {0036-1399},
	journal = {SIAM J. Appl. Math.},
	mrclass = {05.55},
	mrnumber = {268080},
	mrreviewer = {R.\ L.\ Graham},
	pages = {19--24},
	title = {Graphs with monochromatic complete subgraphs in every edge coloring},
	url = {https://doi.org/10.1137/0118004},
	volume = {18},
	year = {1970}}

@incollection{MR278941,
	author = {Specker, E.},
	booktitle = {Logic {C}olloquium '69 ({P}roc. {S}ummer {S}chool and {C}olloq., {M}anchester, 1969)},
	mrclass = {02.70},
	mrnumber = {278941},
	mrreviewer = {James\ D.\ Halpern},
	pages = {439--442},
	publisher = {North-Holland, Amsterdam-London},
	series = {Stud. Logic Found. Math.},
	title = {Ramsey's theorem does not hold in recursive set theory},
	volume = {Vol. 61},
	year = {1971}}

@article{dauriac2023carlsonsimpsonslemmaapplicationsreverse,
	author = {Angl{\`e}s d'Auriac, Paul-Elliot and Liu, Lu and Mignoty, Bastien and Patey, Ludovic},
	journal = {Ann. Pure Appl. Logic},
	number = {9},
	pages = {Paper No. 103287, 16 pp.},
	title = {Carlson-{S}impson's lemma and applications in reverse mathematics},
	volume = {174},
	year = {2023}}

\end{document}